\documentclass[times,sort&compress,3p]{elsarticle} 
\journal{Journal of Multivariate Analysis}
\usepackage{amsmath, amssymb, enumerate, amsthm}

\usepackage{epstopdf,booktabs}

\usepackage{amssymb}
\usepackage{amsfonts}
\usepackage{amsmath}
\usepackage{float}
\usepackage{color}
\newtheorem{lemma}{Lemma}
\newtheorem{remark}{Remark}

\newtheorem{theorem}{Theorem}
\newtheorem{corollary}{Corollary}

\journal{Journal of Multivariate Analysis}

\begin{document}

\begin{frontmatter}

\title{Strongly consistent autoregressive predictors in abstract Banach spaces}

\author[MDRM]{Mar{\'i}a D. Ruiz-Medina\corref{cor}}
\ead{mruiz@ugr.es}
\author[MDRM]{Javier \'Alvarez-Li\'ebana}
\ead{javialvaliebana@ugr.es}

\address[MDRM]{Departamento de Estad\'istica e Investigaci\'on Operativa\\
Universidad de Granada, Campus Fuente Nueva s/n\\18071 Granada, Spain}

\cortext[cor]{Corresponding author}

\begin{abstract} 
This work derives new results on  strong consistent estimation and prediction for autoregressive processes of order  $1$ 
in a separable Banach space $B$. The consistency  results are obtained 
for the component-wise estimator of the autocorrelation operator in the norm of the space $\mathcal{L}(B)$ of bounded linear operators on $B$.   
 The strong consistency of the  associated plug-in predictor then follows in the $B$-norm. 
A Gelfand triple  is defined through the Hilbert space constructed in Kuelbs' lemma \cite{Kuelbs70}. A Hilbert--Schmidt embedding introduces the 
Reproducing Kernel Hilbert space (RKHS), generated by the autocovariance operator,  into the Hilbert space conforming  the Rigged Hilbert space structure.
 This paper extends the work of Bosq \cite{Bosq00} and Labbas and Mourid \cite{Labbas02}. 
\end{abstract}

\end{frontmatter}

\section{Introduction}
\label{sec:1} 

In the last few decades, there has been a growing interest in the statistical analysis of high-dimensional data from the  Functional Data Analysis (FDA) perspective.  The   book by Ramsay and Silverman \cite{RamsaySilverman05}  provides an  overview of FDA techniques adapted from the multivariate data context or specifically formulated for the FDA framework. 
 The   monograph by  Hsing and Eubank \cite{HsingEubank15} introduces  functional analytical tools useful for    the estimation of  random elements in function spaces.  
The book by Horv\'ath and  Kokoszka \cite{HorvathandKokoszka} is mainly concerned with inference based on second order statistics; a central topic in this book is the analysis of functional data exhibiting  dependent  structures    in time and space. The methodological  survey paper by
Cuevas  \cite{Cuevas14} discusses central topics in FDA. 
Recent advances  in the statistical analysis of 
 high-dimensional data from the parametric, semiparametric and nonparametric FDA viewpoints  are collected in the JMVA Special Issue  by Goia and Vieu \cite{GoiaVieu16}.

Linear time series models traditionally arise  for processing temporal linear correlated data. In the FDA context, Bosq's monograph \cite{Bosq00} introduces  linear functional time series theory. The RKHS generated by the autocovariance operator plays a crucial role in the estimation approach presented therein. In particular,  the eigenvectors of the autocovariance operator are considered for projection; see also \cite{Alvarez17}. Its  empirical version is computed when they  are unknown. The resulting  plug-in predictor is obtained as a linear functional of the observations, based  on the empirical approximation of the autocorrelation operator. This approach exploits the Hilbert space structure; its extension to the metric space context, particularly in the Banach space context, relies on a relationship (continuous embeddings) between the  Banach space norm and the RKHS norm induced by the extended autocovariance operator. This is 
in contrast with the nonparametric  regression approach, where semi-metric spaces are usually considered; see, 
e.g.,~\cite{FerratyKeilegom12}, where asymptotic normality is derived, in the regression model with functional response and covariates. In particular,
a linear combination of the observed response values  is  considered,   in the nonparametric local-weighting-based approach.  Here, the weights are defined from an  isotropic kernel depending on the metric or semi-metric of the space in which the regressors take their values; see, e.g., \cite{Ferraty06}, and, in particular, \cite{Ferraty02} in the functional time series framework.  However, the more flexible nonparametric approach also presents some computational drawbacks,  requiring the resolution of several selection problems. For instance, a choice of smoothing parameter and the kernel involved in the definition of the weights should be performed. Real-valued covariates were incorporated in the  semiparametric kernel-based proposal by Aneiros-P\'erez and  Vieu \cite{Aneiros08}, which involves  an extension to the functional partial linear time series framework; 
see also~\cite{AneirosVieu06} on 
semi-functional partial linear regression. Goia and  Vieu~\cite{GoiaVieu15} also adopted a semiparametric approach in their formulation of a two-term Partitioned Functional Single Index Model. Geenens \cite{Geenens11} exploited the alternative provided by semi-metrics to avoid the so-called 
\textit{curse of dimensionality}. 

In a parametric linear framework, functional time series models in Banach spaces were introduced in~\cite{MasPumo10}. Strong mixing conditions and the absolute regularity of Banach-valued autoregressive processes were studied in~\cite{Allam11}. Empirical estimators for Banach-valued autoregressive processes were discussed  in~\cite{Bosq02} where, under some regularity conditions, and for the case of  orthogonal innovations, the empirical mean was shown to be  asymptotically optimal, with respect to almost sure convergence and convergence of order $2$.  The empirical autocovariance operator was also interpreted as a sample mean of an autoregressive process in  a suitable space of linear operators. The extension of these results to the case of weakly dependent innovations was obtained in~\cite{Dehling05}. A strongly consistent  sieve estimator of the autocorrelation operator of a Banach-valued  autoregressive process was considered in \cite{RachediMourid03}.  Limit theorems for a seasonality estimator were given in  \cite{Mourid02} in the case of  Banach autoregressive perturbations. Confidence regions were also obtained for the  seasonality function in the Banach space of continuous functions. An approximation of Parzen's optimal predictor in the RKHS framework was used in  \cite{MokhtariMouri} to predict temporal stochastic processes in Banach spaces. The existence and uniqueness of an almost surely strictly periodically correlated solution to the first order autoregressive  model in Banach spaces was derived in  \cite{Parvardeha}. Under some regularity conditions, limit results were obtained for AR$\mathcal{D}$(1) processes in  \cite{Hajj11}, where $\mathcal{D}[0,1]$ denotes  the Skorokhod space  of right-continuous functions on $[0,1]$ having a left limit at all $t \in [0,1]$. Conditions for the existence of  strictly stationary solutions of ARMA equations in Banach spaces with independent and identically distributed noise innovations can be found in~\cite{Spangenberg13}.   

In deriving strong consistency results for ARB$(1)$  component-wise estimators and predictors,  Bosq \cite{Bosq00} restricted his attention to the case  of the  Banach space   $\mathcal{C}[0,1]$ of continuous functions on $[0,1]$ with the sup norm.
 Labbas and  Mourid   \cite{Labbas02} considered  an ARB$(1)$  context, for an arbitrary real separable Banach space $B$,  under the construction of a Hilbert space  $\widetilde{H}$, where $B$ is 
 continuously embedded as in Kuelbs'   lemma~\cite{Kuelbs70}.
Assuming the existence of a continuous extension to
$\widetilde{H}$ of the  autocorrelation operator $\rho \in \mathcal{L}(B),$  they proved the 
strong consistency of the formulated  component-wise estimator of  
$\rho $   and of its associated plug-in predictor in the norms of  $\mathcal{L}(\widetilde{H})$  and   $\widetilde{H}$,
respectively. 

The linear time series framework in Banach spaces studied here is motivated by the statistical analysis of temporal correlated  
 functional data in nuclear spaces, arising notably in the observation of the solution to stochastic differential or fractional  pseudodifferential equations; see, e.g., \cite{Anh16a,Anh16b}. The scales of Banach spaces constituted by  fractional Sobolev and Besov spaces play a central role in the context of nuclear spaces. Continuous (nuclear)  embeddings usually connect the elements of these scales; see, e.g.,~\cite{Triebel83}. In this paper  a Rigged-Hilbert-Space structure is defined, involving the separable Hilbert space $\widetilde{H}$ appearing in the construction of Kuelbs' lemma~\cite{Kuelbs70}. A key assumption here is the existence of a   continuous (Hilbert--Schmidt class) embedding  introducing the RKHS associated with the extended autocovariance operator of the ARB$(1)$  process, into the Hilbert space generating the Gelfand triple, equipped with a finer topology than the $B$-topology.  Under this scenario,    
  strong consistency results are derived in the space $\mathcal{L}(B)$ of bounded linear operators on $B$, considering an abstract separable Banach space framework.
    
This paper is structured as follows. Background material and notation are first given in Section~\ref{sec:2}. Section~\ref{sec:221} states the basic assumptions and key lemmas which are then proved in Section~\ref{sec:4}. This paper's main strong consistency result is derived in Section~\ref{sec:3}, and examples are presented in Section~\ref{examples}. Closing comments are in Section~\ref{fc}. The results are illustrated numerically in an Online Supplement under  the scenario described in Section~\ref{examples}.

\section{Preliminaries}
\label{sec:2}
Let $\left(B, \left\| \cdot \right\|_B \right)$ be a  real separable Banach space equipped with the norm $\left\| \cdot \right\|_B$ and let $\mathcal{L}^{2}_{B}(\Omega,\mathcal{A}, P)$ be the space of zero mean $B$-valued random variables $X$ such that 
$$
\sqrt{\int_{B}\|X\|_{B}^{2}\, dP}<\infty.
$$ 

Let $X = \{ X_n: n \in \mathbb{Z} \}$ be a zero mean $B$-valued stochastic process  on the probability space $\left( \Omega, \mathcal{A}, P \right)$ such that, for all $n \in \mathbb{Z}$,
\begin{equation}
X_n = \rho \left( X_{n-1} \right) + \varepsilon_n,\label{state_equation_Banach}
\end{equation}
\noindent where $\rho \in \mathcal{L}(B)$ denotes the autocorrelation operator of $X$; see  \cite{Bosq00}.   
In Eq.~\eqref{state_equation_Banach}, the $B$-valued innovation  process $\varepsilon = \{ \varepsilon_n: n \in \mathbb{Z} \}$ on $\left( \Omega, \mathcal{A}, P \right)$ is assumed to be a strong white noise  uncorrelated with the random initial condition. Thus $\varepsilon$ is a zero mean Banach-valued stationary process with iid components and $\sigma_{\varepsilon}^{2} = {\rm E} \{ \left\| \varepsilon_n \right\|_{B}^{2} \} < \infty $ for all $n\in \mathbb{Z}$.

Assume that there exists an integer $j_0 \in  \{ 1, 2, \ldots\}$ such that 
\begin{equation}
\| \rho^{j_0}  \|_{\mathcal{L}\left(B \right)} < 1.\label{A1}
\end{equation}
Then Eq.~\eqref{state_equation_Banach} admits a unique strictly stationary solution with $\sigma_{X}^{2}={\rm E} \{ \left\| X_n \right\|_{B}^{2} \} < \infty,$ i.e., belonging to $\mathcal{L}^{2}_{B}(\Omega,\mathcal{A}, P),$    given by
$ X_n = \sum_{j\in \mathbb{Z}}^{\infty} \rho^j  ( \varepsilon_{n-j} ),$  for each $n \in \mathbb{Z}$; see  \cite{Bosq00}.    
Under \eqref{A1}, the autocovariance operator $C$ of an ARB$(1)$  process $X$ is defined from the autocovariance operator of $X_{0}\in \mathcal{L}_{B}^{2}(\Omega, \mathcal{A}, P)$  
  as
$C ( x^{\star} )=  {\rm E}\{ x^{\star}(X_{0})X_{0}\}$
for all $x^{\star} \in B^{\star}$, and the cross-covariance operator  $D$ is given by
$D  ( x^{\star}  )  = {\rm E} \{  x^{\star}(X_0)X_{1} \}$ for all $ x^{\star} \in B^{\star}$.
Since $C$ is assumed to be a nuclear operator, then as per Eq.~(6.24) on p.~156 of \cite{Bosq00}, there exists a sequence $\{ x_j: j \geq 1 \} \subset B$ such that, for every $x^{\star} \in B^{\star}$,
\begin{equation*}
C( x^{\star}) = \displaystyle \sum_{j=1}^{\infty}  x^{\star} ( x_j ) x_j, \quad \displaystyle \sum_{j=1}^{\infty}  \| x_j  \|_{B}^{2} < \infty. 
\end{equation*}
If $D$ is also assumed to be a nuclear operator, then, as per Eq.~(6.23) on p.~156  of \cite{Bosq00}, there exist  sequences $\{ y_j: j \geq 1 \} \subset B$ and $\{ x_{j}^{\star \star}: j \geq 1 \} \subset B^{\star \star}$ 
such that, for every $x^{\star} \in B^{\star},$
\begin{equation*}
D( x^{\star}) = \displaystyle \sum_{j=1}^{\infty}x_{j}^{\star \star}
(x^{\star})y_j,\quad \displaystyle \sum_{j=1}^{\infty}   \| x_{j}^{\star \star}  \|  \times \| y_{j}  \| < \infty .
\end{equation*}
From Eqs.~(6.45) and (6.58) on pp. 164--168 of \cite{Bosq00}, empirical estimators of $C$ and $D$ are respectively 
 given, for all $x^{\star}\in B^{\star}$ and any integer $n\geq 2$, by 
$$
C_n ( x^{\star}) = \frac{1}{n} \displaystyle \sum_{i=0}^{n-1}  x^{\star} \left( X_i \right) \left(X_i \right), 
 \quad D_n ( x^{\star}) = \frac{1}{n-1} \displaystyle \sum_{i=0}^{n-2}  x^{\star} \left( X_i \right) \left(X_{i+1}\right).
 $$ 

Lemma~2.1 in \cite{Kuelbs70}, recalled just below,  plays a key role in our approach.
\begin{lemma}
\label{lemma:1}
If $B$ is a real separable Banach space with norm $\left\| \cdot \right\|_{B},$
then there exists an inner product $\langle \cdot,\cdot\rangle_{\widetilde{H}} $ on $B$ such that the norm $\left\| \cdot \right\|_{\widetilde{H}},$ generated
by $\langle \cdot, \cdot\rangle_{\widetilde{H}} ,$ is weaker than $\left\| \cdot \right\|_{B}.$ The completion of $B$ under the   norm $\left\| \cdot \right\|_{\widetilde{H}}$ defines  the Hilbert space $\widetilde{H},$   where $B$ is continuously embedded.  
\end{lemma}
 Denote by 
$\{ x_n: n \geq 1 \} \subset B$ a dense sequence in $B$ and by   $\{ F_n: n \geq 1 \} \subset B^{\star}$ a sequence of bounded linear functionals on $B$ satisfying
\begin{equation}
F_n \left( x_n \right) = \left\| x_n \right\|_B, \quad \left\| F_n \right\| =  1,\label{normFn}
\end{equation}
such that, for all $x \in B$,  
\begin{equation} \left\| x \right\|_B  = \displaystyle \sup_{n \geq 1} \left| F_n (x) 
\right|.\label{normB}
\end{equation}
The inner product  $\langle \cdot,\cdot\rangle_{\widetilde{H}},$ and associated norm,  in Lemma~\ref{lemma:1}, is defined, for all $x,y \in \widetilde{H}$, by
$$
\langle x,y \rangle_{\widetilde{H}} =  \sum_{n=1}^{\infty} t_n F_n (x) F_n (y),
$$
while for all $x \in B$,
\begin{equation}
\left\| x \right\|_{\widetilde{H}}^{2} = \displaystyle \sum_{n=1}^{\infty} t_n \left\{F_n (x) \right\}^2 \leq \left\| x \right\|_{B}^{2}, 
\label{ineq_norm}
\end{equation}
where  $\{ t_n: n  \geq 1 \} $ is a sequence of positive numbers  summing up to $1$. 

\section{Main assumptions and preliminary results}
\label{sec:221}

In view of Lemma~\ref{lemma:1}, for every $n\in \mathbb{Z},$ $X_{n}\in B\hookrightarrow\widetilde{H}$ satisfies almost surely, for all $n\in\mathbb{Z}$,
\begin{equation*}
X_{n}\underset{\widetilde{H}}{=}\sum_{j=1}^{\infty} \langle X_n, v_{j} \rangle_{\widetilde{H}} v_{j},
\end{equation*}
where $\{v_{j}:  j \geq 1 \}$ is any orthonormal basis of $\widetilde{H}.$  The 
 trace auto-covariance operator 
 $$
 C = {\rm E} \left\{\left(\sum_{j=1}^{\infty} \langle X_n, v_{j} \rangle_{\widetilde{H}} v_{j}\right)\otimes \left(\sum_{j=1}^{\infty} \langle X_n, v_{j} \rangle_{\widetilde{H}} v_{j}\right)\right\}
 $$ 
 of the extended ARB$(1)$  process is a trace operator on $\widetilde{H}$ admitting a diagonal spectral representation in terms of its eigenvalues 
 $\{C_{j}:  j \geq 1 \}$ and eigenvectors $\{\phi_{j}:  j \geq 1 \}$ that provide an orthonormal system in $\widetilde{H}$.   
In what follows, the following identities in $\widetilde{H}$ will be considered, for the extended version of ARB$(1)$  process $X$.
For arbitrary $f,h\in \widetilde{H},$ 
\begin{eqnarray}
&&C(f)\underset{\widetilde{H}}{=}\sum_{j=1}^{\infty}C_{j} \langle f,\phi_{j} \rangle_{\widetilde{H}} \, \phi_{j},
\label{cest}\\
&&D(h)\underset{\widetilde{H}}{=}\sum_{j=1}^{\infty}\sum_{k=1}^{\infty} \langle 
D(\phi_{j}),\phi_{k} \rangle_{\widetilde{H}} \langle h,\phi_{j}\rangle_{\widetilde{H}}\, \phi_{k},\nonumber\\
&&C_{n}(f)\underset{\widetilde{H}\ {\rm a.s.}}{=}\displaystyle \sum_{j=1}^{n} C_{n,j} \langle f,\phi_{n,j}\rangle_{\widetilde{H}}\, \phi_{n,j}, \label{empcn}\\
&&C_{n,j} \underset{{\rm a.s.}}{=} \frac{1}{n} \displaystyle \sum_{i=0}^{n-1} X_{i,n,j}^{2}, \ X_{i,n,j} = \langle X_i, \phi_{n,j} \rangle_{\widetilde{H}},\  C_{n}(\phi_{n,j})\underset{\widetilde{H}\ {\rm a.s.}}{=}C_{n,j} \, \phi_{n,j} ,
\nonumber\\
&&D_{n}(h)\underset{\widetilde{H}\ {\rm a.s.}}{=}\sum_{j=1}^{\infty}\sum_{k=1}^{\infty}\langle D_{n}(\phi_{n,j}), \phi_{n,k}\rangle_{\widetilde{H}}\langle h,\phi_{n,j}\rangle_{\widetilde{H}} \, \phi_{n,k},
\label{empdn}
\end{eqnarray}
\noindent where, for arbitrary integer $n \geq 2,$ $\{\phi_{n,j}: j \geq 1 \}$  is a complete orthonormal  system in $\widetilde{H},$
and $C_{n,1}\geq \dots\geq C_{n,n}\geq 0= C_{n,n+1}=
C_{n,n+2}=\cdots$

The following assumption plays a crucial role in the derivation of the main results in this paper.

\medskip
\noindent
 \textbf{Assumption~A1.} $\| X_{0}\|_{B}$ is almost surely bounded, and the eigenspace $V_{j}$ associated with 
$C_{j}>0$ in \eqref{cest} is 
one-dimensional for every  integer $ j \geq 1$.

\bigskip
Under {Assumption~A1}, we can define the following quantities:
\begin{equation}
a_1 = 2 \sqrt{2}\,  \frac{1}{C_1 - C_2}, \quad  a_j = 2 \sqrt{2} \, \displaystyle \max \left(\frac{1}{C_{j-1} - C_j}, \frac{1}{C_j - C_{j+1}} \right),\quad j\geq 2. \label{a_j}
\end{equation}
\begin{remark}
\label{rem0}
\em
This assumption can be relaxed to  consider multidimensional eigenspaces by redefining the quantities $a_1, a_2, \ldots$ as the quantities $c_1, c_2, \ldots$ given in Lemma 4.4 of \cite{Bosq00}.
\end{remark}

\medskip
\noindent  
\textbf{Assumption~A2}. Let $k_{n}$ be such that
$ C_{n,k_{n}}>0$ a.s., and both $k_{n}\to \infty$ and ${k_{n}}/{n}\to 0$ as $n\to \infty$.

\medskip
\begin{remark}
\label{remark:2}
\em
Consider
\begin{equation}\Lambda_{k_{n}}=\sup_{j \in \{ 1, \ldots,  k_{n}\} }(C_{j}-C_{j+1})^{-1}.\label{uee}
\end{equation}
Then for sufficiently  large $n$, we have
\begin{equation*}
k_{n}<  C_{k_n}^{-1} < \frac{1}{C_{k_n} - C_{k_n + 1}} < a_{k_n} < \Lambda_{k_{n}}<\displaystyle \sum_{j=1}^{k_n}
 a_{j}.
\end{equation*}
 \end{remark}

\noindent 
\textbf{Assumption~A3.}  As $k \to \infty$,
$$
\sup_{x\in B, \ \|x\|_{B}\leq 1}\left\|\rho(x)-\sum_{j=1}^{k}\langle\rho(x),\phi_{j}\rangle_{\widetilde{H}}\phi_{j}\right\|_{B}\to 0.
$$

\bigskip 
\noindent 
\textbf{Assumption~A4.} 
The constants $\{C_{j}:  j \geq 1 \}$ are such that the inclusion of $\mathcal{H}(X)$ into  $\widetilde{H}^{\star}$ is continuous, i.e.,
$\mathcal{H}(X) \hookrightarrow \widetilde{H}^{\star},
$ where $\hookrightarrow $ denotes the continuous embedding.

\medskip
Let us consider the closed subspace $H$ of $B$ with the norm induced by
the inner product $\langle \cdot,\cdot  \rangle_{H}$  defined as follows:
\begin{equation}
H = \left\{x\in B: \sum_{n=1}^{\infty}\left\{F_{n}(x)\right\}^{2}<\infty\right\},\quad \forall_{ f,g\in H} \; \;  \langle f,g\rangle_{H}=\sum_{n=1}^{\infty} F_{n}(f)F_{n}(g). \label{spaces}
\end{equation}
Then $H$ is continuously embedded into $B$ and the following remark provides the isometric isomorphism   established by the Riesz Representation Theorem between the spaces $\widetilde{H}$ and its dual $\widetilde{H}^{\star}.$  

\begin{remark}
\label{rem1}
\em
Let $f^{\star },g^{\star}\in \widetilde{H}^{\star},$ and $f,g\in \widetilde{H},$ such that, for every integer $n\geq 1$, consider $F_{n}(f^{\star})= 
    \sqrt{t_{n}} \, F_{n}(\widetilde{f}),$ $F_{n}(g^{\star})= 
    \sqrt{t_{n}}\, F_{n}(\widetilde{g}),$  and $F_{n}(\widetilde{f})=\sqrt{t_{n}} \, F_{n}(f),$ $F_{n}(\widetilde{g})=\sqrt{t_{n}} \, F_{n}(g)$ for certain $\widetilde{f},\widetilde{g}\in H.$
Then the following identities hold:
\begin{eqnarray}
\langle f^{\star}, g^{\star }
\rangle_{\widetilde{H}^{\star}}=
    \sum_{n=1}^{\infty}\frac{1}{t_{n}}F_{n}(f^{\star})F_{n}(g^{\star })=\sum_{n=1}^{\infty}\frac{1}{t_{n}}\sqrt{t_{n}}\sqrt{t_{n}}F_{n}(\widetilde{f})F_{n}(\widetilde{g})=\langle \widetilde{f},\widetilde{g}\rangle_{H} =\sum_{n=1}^{\infty}t_{n}F_{n}(f)F_{n}(g)=\langle f,g\rangle_{\widetilde{H}}.\nonumber \end{eqnarray}

\end{remark}

\begin{lemma}
\label{lemmembeddhold}
Under {Assumption~A4},   the following continuous embeddings hold:
\begin{equation} 
\mathcal{H}(X) \hookrightarrow \widetilde{H}^{\star}\hookrightarrow  B^{\star}\hookrightarrow H \hookrightarrow B\hookrightarrow \widetilde{H} \hookrightarrow \{\mathcal{H}(X)\}^{\star },
  \label{embedding}
\end{equation}
where
\begin{align*}
\widetilde{H} &=\left\{x\in B: \sum_{n=1}^{\infty}
t_{n}\left\{F_{n}(x)\right\}^{2}<\infty\right\}, \quad \forall_{f,g\in \widetilde{H}} \; \; \langle f,g\rangle_{\widetilde{H}}=\sum_{n=1}^{\infty}t_{n} F_{n}(f)F_{n}(g), \\
H &= \left\{x\in B: \sum_{n=1}^{\infty}\left\{F_{n}(x)\right\}^{2}<\infty\right\}, \quad \forall_{f,g\in H} \;\; \langle f,g\rangle_{H}=\sum_{n=1}^{\infty} F_{n}(f)F_{n}(g), \\
\widetilde{H}^{\star }&= \left\{x\in B: \sum_{n=1}^{\infty}
\frac{1}{t_{n}}\left\{F_{n}(x)\right\}^{2}<\infty\right\}, \quad \forall_{f,g\in \widetilde{H}^{\star}} \; \; \langle f,g\rangle_{\widetilde{H}^{\star }}=\sum_{n=1}^{\infty} F_{n}(f)F_{n}(g)/t_n, \\
\mathcal{H}(X)&=  \{x\in \widetilde{H}; \ \langle C^{-1}(x),x\rangle_{\widetilde{H}}<\infty \}, \quad \forall_{f,g\in C^{1/2}(\widetilde{H})} \;\; 
\langle f,g\rangle_{\mathcal{H}(X)}=\langle C^{-1}(f),g\rangle_{\widetilde{H}}, \\ 
\left\{ \mathcal{H}(X)\right\}^{\star}&=  \{x\in \widetilde{H}; \ \langle C(x),x\rangle_{\widetilde{H}}<\infty \} , \quad \forall_{f,g\in C^{-1/2}(\widetilde{H})} \;\; 
\langle f,g\rangle_{\{\mathcal{H}(X)\}^{\star }}=\langle C(f),g\rangle_{\widetilde{H}}. 
 \end{align*}
\end{lemma}

\begin{proof}
 Let us consider  the following inequalities, valid for all $x\in B$ such that $x\in \widetilde{H}^{\star}$:
\begin{eqnarray}
&&\|x\|_{\widetilde{H}}=
\sqrt{\sum_{n=1}^{\infty}t_{n}\left\{F_{n}(x)\right\}^{2}}
\leq 
\|x\|_{B}=\sup_{n \geq 1} |F_{n}(x)|,\quad 
\nonumber\\
&&\|x\|_{B}=\sup_{n\geq 1}|F_{n}(x)|\leq 
\sqrt{\sum_{n=1}^{\infty}\left\{F_{n}(x)\right\}^{2}}=\|x\|_{H}
\leq \sum_{n=1}^{\infty}|F_{n}(x)|=\|x\|_{B^{\star}},
\nonumber\\
&&\|x\|_{B^{\star}}=\sum_{n=1}^{\infty}|F_{n}(x)|\leq \sqrt{\sum_{n=1}^{\infty} \frac{1}{t_{n}}\left\{F_{n}(x)\right\}^{2}}=\|x\|_{\widetilde{H}^{\star}}.
 \label{embdd}\end{eqnarray} 
Under {Assumption~A4} (see also Remark~\ref{rem1}), for every $f\in C^{1/2}(\widetilde{H})=\mathcal{H}(X)$, we have   \begin{equation}
\|f\|_{\mathcal{H}(X)}=\sqrt{\langle C^{-1}(f),f\rangle_{\widetilde{H}}}\geq \|f\|_{\widetilde{H}^{\star}}=\sqrt{\sum_{n=1}^{\infty} \frac{1}{t_{n}}\left\{F_{n}(x)\right\}^{2}}.\label{embdd2}
\end{equation}
From Eqs.~(\ref{embdd}) and (\ref{embdd2}), the  inclusions in  (\ref{embedding}) are continuous.  Thus the proof
is complete. \hfill $\Box$
\end{proof}

\bigskip
It is well known that the set $\{\phi_{j}:  j \geq 1 \}$ is also an orthogonal system in $\mathcal{H}(X).$
Futhermore, under {Assumption~A4}, from Lemma~\ref{lemmembeddhold},
$\{\phi_{j}:  j \geq 1 \}\subset \mathcal{H}(X)\hookrightarrow \widetilde{H}^{\star}\hookrightarrow  B^{\star}\hookrightarrow H$.
Therefore, from  (\ref{spaces}), for every $ j \geq 1,$
\begin{equation}\|\phi_{j}\|_{H}^{2}=\sum_{m=1}^{\infty} \{F_{m}(\phi_{j}) \}^{2}<\infty.\label{evHn}
\end{equation}

\noindent The following assumption is now considered on the norm (\ref{evHn}).

\bigskip
\noindent 
\textbf{Assumption~A5.}
The continuous embedding $i_{\mathcal{H}(X),H}: \mathcal{H}(X)\hookrightarrow H $ belongs to the Hilbert--Schmidt  class, i.e., 
   $\sum_{j=1}^{\infty}\|\phi_{j}\|_{H}^{2}<\infty $. Let $\{ F_{m}: m\geq 1\}$ be defined as in Lemma~\ref{lemma:1}. 
{Assumption~A5} leads to 
\begin{equation}\sum_{j=1}^{\infty}\langle i_{\mathcal{H}(X),H}(\phi_{j}),
i_{\mathcal{H}(X),H}(\phi_{j})\rangle_{H}=\sum_{j=1}^{\infty}\sum_{m=1}^{\infty}\{F_{m}(\phi_{j}) \}^{2}=\sum_{m=1}^{\infty}N_{m}<\infty , \label{inclusion2}
\end{equation}
where, in particular, from \eqref{inclusion2},
\begin{equation}
\label{cfm}
N_{m} =\sum_{j=1}^{\infty} \{F_{m}(\phi_{j}) \}^{2}
<\infty, \quad  \sup_{m \geq 1}N_{m}=N<\infty, \quad 
V = \displaystyle \sup_{j \geq 1}  \| \phi_j \|_B \leq \sum_{j=1}^{\infty}\sum_{m=1}^{\infty} \{F_{m}(\phi_{j}) \}^{2} <\infty. 
\end{equation}

The following preliminary results are deduced from Theorem~4.1 (pp. 98--99), Corollary~4.1 (pp. 100--101), and Theorem~4.8 (pp. 116--117) in \cite{Bosq00}.

\begin{lemma}
\label{theorem2} 
Under  {Assumption~A1},  the following identities hold, for any standard 
AR$\widetilde{H}(1)$ process, e.g., the extension to $\widetilde{H}$ of ARB$(1)$  process $X$ satisfying Eq.~\eqref{state_equation_Banach}, as  $n\to \infty$: 
\begin{equation}
\left\| C_n - C \right\|_{\mathcal{S}(\widetilde{H})} = \mathcal{O} [ \{ {\ln(n) }/{n}  \}^{1/2} ]~{a.s.}, \quad 
\left\| D_n - D \right\|_{\mathcal{S}(\widetilde{H})} = \mathcal{O}[  \{  {\ln(n) }/{n} \}^{1/2}  ]~{a.s.}\label{fth2}
\end{equation}
Here, $\to_{{\rm a.s.}}$ denotes   almost surely convergence, and  $\left\| \cdot \right\|_{\mathcal{S}(\widetilde{H})}$ is the  norm in the Hilbert space $\mathcal{S}(\widetilde{H})$ of  Hilbert--Schmidt operators on $\widetilde{H}$, i.e., the subspace of compact operators $\mathcal{A}$ such that $\sum_{j=1}^{\infty}
 \langle \mathcal{A}^{\star}\mathcal{A}(\varphi_{j}),\varphi_{j} \rangle_{\widetilde{H}}<\infty,$ for any orthonormal basis $\{\varphi_{j}:  \ j \geq 1 \}$ of  $\widetilde{H}.$  
\end{lemma}

\begin{lemma}
\label{lem1}
Under {Assumption~A1}, let  $\{C_{j}:  j \geq 1 \}$ and  $\{C_{n,j}:  j \geq 1 \}$  in \eqref{cest} and  \eqref{empcn}, respectively. Then, as $n \to \infty$, 
$\{ {n}/{\ln(n)} \}^{1/2} 
\sup_{j \geq 1} | C_{n,j}- C_j |\to_{{\rm a.s.}} 0.$ 
\end{lemma}

The following lemma is Corollary~4.3 on p.~107 of \cite{Bosq00}.

\begin{lemma}
\label{lemmanew} 
Under \textit{Assumption} \textbf{A1},  
consider $\Lambda_{k_{n}}$ in \eqref{uee} satisfying 
 $\Lambda_{k_{n}}=o [\{ {n}/{\ln(n)} \}^{1/2} ],$ as $n\rightarrow \infty.$   Then, as $n \to \infty$,
\begin{equation*}
\sup_{j \in \{ 1, \ldots, k_{n}\}}\|\phi_{n,j}^{\prime }-\phi_{n,j}\|_{\widetilde{H}}\to_{{\rm a.s.}} 0,
\end{equation*}
where, for every integer  $j\geq 1$ and  $n\geq 2$, 
 $\phi_{n,j}^{\prime }= {\rm sgn}\langle \phi_{n,j} , \phi_{j} \rangle_{\widetilde{H}}\phi_{j},$ ${\rm sgn} \langle \phi_{n,j}, \phi_j \rangle_{\widetilde{H}}= \mathbf{1}_{\langle \phi_{n,j}, \phi_j \rangle_{\widetilde{H}}\geq 0}-\mathbf{1}_{\langle \phi_{n,j}, \phi_j \rangle_{\widetilde{H}}< 0},$ with   $\mathbf{1}$ denoting an indicator function.
\end{lemma}
An upper bound for $\left\|c\right\|_{B\times B}= \|\sum_{j=1}^{\infty}C_{j}\phi_{j}
\otimes \phi_{j} \|_{B\times B}$ is  obtained next.

\begin{lemma}
\label{lemma:3}
Under {Assumption~A5}, the following inequality holds:
$$
\left\| c  \right\|_{B \times B}
 = \sup_{n,m \geq 1} \left| C \left( F_n \right) \left(F_m \right) \right|
 \leq N\left\| C \right\|_{\mathcal{L}(\widetilde{H})},
 $$
 where $N$  has been introduced in \eqref{cfm}, $\mathcal{L}(\widetilde{H})$ denotes the space of bounded linear operators on $\widetilde{H},$ and $\left\| \cdot \right\|_{\mathcal{L}(\widetilde{H})}$ the usual uniform norm on such a space.
\end{lemma}

Let us consider the following notation:
$$
c\underset{\widetilde{H}\otimes \widetilde{H}}{=}\sum_{j=1}^{\infty}C_{j}\phi_{n,j}^{\prime}\otimes \phi_{n,j}^{\prime}\underset{\widetilde{H}\otimes \widetilde{H}}{=}\sum_{j=1}^{\infty}C_{j}\phi_{j}\otimes \phi_{j},\quad c_{n}\underset{\widetilde{H}\otimes \widetilde{H}}{=}\sum_{j=1}^{\infty}C_{n,j}\phi_{n,j}\otimes \phi_{n,j},
$$
\begin{eqnarray}
c-c_{n}\underset{\widetilde{H}\otimes \widetilde{H}}{=}\sum_{j=1}^{\infty}C_{j}\phi^{\prime }_{n,j}\otimes \phi^{\prime }_{n,j}- \sum_{j=1}^{\infty}C_{n,j}\phi_{n,j}\otimes \phi_{n,j}.
\label{eqkernelcn}
\end{eqnarray}

\begin{remark}
\label{rekernelrhhsnorm}
\em
From Lemma~\ref{theorem2}, for sufficiently large $n$, 
there exist positive constants $K_{1}$ and $K_{2}$ such that, for all $ \varphi \in \widetilde{H}$,
\begin{equation*}
K_{1}\langle C(\varphi ),\varphi\rangle_{\widetilde{H}}\leq \langle C_{n}(\varphi ),\varphi\rangle_{\widetilde{H}}\leq K_{2}\langle C(\varphi ),\varphi\rangle_{\widetilde{H}}.
\end{equation*}
In particular, for every $x\in \mathcal{H}(X)=C^{1/2}(\widetilde{H}),$ considering  $n$ sufficiently large, we find
\begin{eqnarray}
\frac{1}{K_{1}}\langle C^{-1}(x),x\rangle_{\widetilde{H}}\geq \langle C_{n}^{-1}(x),x\rangle_{\widetilde{H}}\geq \frac{1}{K_{2}}\langle C^{-1}(x),x\rangle_{\widetilde{H}} \quad 
 \Leftrightarrow \quad \frac{1}{K_{1}} \, \|x\|_{\mathcal{H}(X)}^{2}\geq \langle C_{n}^{-1}(x),x\rangle_{\widetilde{H}}\geq \frac{1}{K_{2}} \, \|x\|_{\mathcal{H}(X)}^{2}.
\label{eqenbb}
\end{eqnarray}
Eq.~\eqref{eqenbb} means that,   for sufficiently large $n$, the norm of the RKHS $\mathcal{H}(X)$ of $X$ is equivalent to the norm of the RKHS generated by $C_{n},$ with spectral kernel $c_{n}$ given in \eqref{eqkernelcn}. 
\end{remark}

\begin{lemma}
\label{lemsv}
Under {Assumptions A1--A5}, consider $\Lambda_{k_{n}}$ in \eqref{uee} satisfying \begin{equation}\sqrt{k_{n}}\Lambda_{k_{n}}=
o\big\{\sqrt{{n}/{\ln(n)}}\big\}
\label{eqcondindisp}
\end{equation} 
as $n \to \infty$, 
where $k_{n}$ has been introduced in {Assumption~A2}.
  The following almost sure inequality then holds:
\begin{multline*}
\left\| c-c_{n}  \right\|_{B \times B}
 \leq \max(N, \sqrt{N})\left[
\|C-C_{n}\|_{\mathcal{L}(\widetilde{H})}+2\max\left(\sqrt{\|C\|_{\mathcal{L}(\widetilde{H})}},\sqrt{\|C_{n}\|_{\mathcal{L}(\widetilde{H})}}\right)\left[\sup_{\ell \geq 1}\sup_{m\geq 1} |F_{\ell}(\phi^{\prime}_{n,m}) |\right]\right.\nonumber\\
\left.\times \sqrt{k_{n}8\Lambda_{k_{n}}^{2}
\|C_{n}-C\|_{\mathcal{L}(\widetilde{H})}^{2}+\sum_{m=k_{n}+1}^{\infty}\|\phi_{n,m}-
\phi_{n,m}^{\prime}\|_{\widetilde{H}}^{2}}\right].
\end{multline*}
Therefore, $\left\| c-c_{n}  \right\|_{B \times B}\to_{{\rm a.s.}} 0$ as $n\to \infty. $
\end{lemma}

\begin{lemma}
\label{leminfinit}
For a standard ARB$(1)$  process satisfying Eq.~\eqref{state_equation_Banach},  under  {Assumptions A1--A5}, consider $\Lambda_{k_{n}}$ in \eqref{uee} such that \begin{equation}C_{k_{n}}^{-1}\sqrt{k_{n}}\Lambda_{k_{n}}=
o\big\{\sqrt{{n}/{\ln(n)}}\big\}\label{orl8}
\end{equation}\noindent as $n \to \infty$, 
where $k_{n}$ has been introduced in {Assumption~A2}.
  The following  inequality then is established
\begin{multline}
\sup_{j \in \{ 1, \ldots, k_{n}\}}\left\| \phi_{n,j} - \phi_{n,j}^{\prime} \right\|_B \leq \frac{2}{C_{k_{n}}} {\Bigg [}
\max(N, \sqrt{N}) {\Bigg [} 
\|C-C_{n}\|_{\mathcal{L}(\widetilde{H})} 
 +2\max\left(\sqrt{\|C\|_{\mathcal{L}(\widetilde{H})}},\sqrt{\|C_{n}\|_{\mathcal{L}(\widetilde{H})}}\right)\left\{\sup_{\ell \geq 1}\sup_{m \geq 1} |F_{\ell}(\phi^{\prime}_{n,m}) |\right\} \\
 \times \sqrt{k_{n}8\Lambda_{k_{n}}^{2}
\|C_{n}-C\|_{\mathcal{L}(\widetilde{H})}^{2}+\sum_{m=k_{n}+1}^{\infty}\|\phi_{n,m}-
\phi_{n,m}^{\prime}\|_{\widetilde{H}}^{2}} {\Bigg ] } \\
 +\sup_{j \in \{ 1, \ldots, k_{n}\}}\|\phi_{n,j} - \phi_{n,j}^{\prime}
\|_{\widetilde{H}}N\|C\|_{\mathcal{S}(\widetilde{H})}+ 
V\|C-C_{n}\|_{\mathcal{S}(\widetilde{H})} {\Bigg ]} \quad \mbox{{a.s.}}\label{eqlem7}
\end{multline}
Therefore, under (\ref{orl8}), $\sup_{j \in \{ 1, \ldots, k_{n}\}} \| \phi_{n,j} - \phi_{n,j}^{\prime}  \|_B\to_{{\rm a.s.}} 0,$ as $n\to \infty. $ 
\end{lemma}
\begin{lemma}
\label{ultlemma}
Under {Assumption~A3}, if $\sum_{j=1}^{k_{n}}\|\phi_{n,j}-\phi_{n,j}^{\prime }\|_{B}\to_{{\rm a.s.}} 0$ as $n\to \infty,$ then also
\begin{equation}
\sup_{x\in B, \, \|x\|_{B}\leq 1}\left\|\rho(x)-\sum_{j=1}^{k_{n}}\langle \rho(x),\phi_{n,j}\rangle_{\widetilde{H}}\phi_{n,j}\right\|_{B}\to_{{\rm a.s.}} 0.
\label{lemnornrho}
\end{equation}
\end{lemma}

\begin{remark}
\label{remprev}
\em
Under the conditions of Lemma~\ref{leminfinit}, Eq.~\eqref{lemnornrho} holds as soon as 
$$
C_{k_{n}}^{-1}k_{n}^{3/2}\Lambda_{k_{n}}=o\big\{\sqrt{{n}/{\ln(n)}}\big\}.
$$
\end{remark}

Let us know consider the projection operators defined, for all $x\in B\subset \widetilde{H}$, by
\begin{eqnarray}
\widetilde{\Pi}^{k_n}  \left( x \right)=  \displaystyle \sum_{j=1}^{k_n} \langle x, \phi_{n,j} \rangle_{\widetilde{H}} \phi_{n,j},\quad 
\Pi^{k_n}\left(x\right) = \displaystyle \sum_{j=1}^{k_n} \langle x, \phi_{n,j}^{\prime } \rangle_{\widetilde{H}} \phi_{n,j}^{\prime }.
\label{proy}
\end{eqnarray}

\begin{remark}
\label{remprev2}
\em
Under the conditions of Remark~\ref{remprev},   let 
$$
\widetilde{\Pi}^{k_{n}}\rho\widetilde{\Pi}^{k_{n}}=\sum_{j=1}^{k_{n}}\sum_{p=1}^{k_{n}}\langle \rho(\phi_{n,j}),\phi_{n,p}
\rangle_{\widetilde{H}}\phi_{n,j}\otimes \phi_{n,p}.
$$
Then, as $n\to \infty$,
 \begin{equation*}
\sup_{x\in B,\ \|x\|_{B}\leq 1}\left\|\rho(x)-\sum_{j=1}^{k_{n}}\sum_{p=1}^{k_{n}}\langle x,\phi_{n,j}\rangle_{\widetilde{H}}\langle \rho(\phi_{n,j}),\phi_{n,p}
\rangle_{\widetilde{H}}\phi_{n,p}\right\|_{B}\to_{{\rm a.s.}} 0.
\end{equation*}
\end{remark}

\section{Proofs of the lemmas}
\label{sec:4}

\subsection{Proof of Lemma~\ref{lemma:3}}

Applying the Cauchy--Schwarz inequality, we have, for all integers $k, \ell \geq 1$, 
\begin{align}
|C(F_{k},F_{\ell})| &= \left|\sum_{j=1}^{\infty}C_{j}F_{k}(\phi_{j})F_{\ell}(\phi_{j})\right| \leq 
\sqrt{\sum_{j=1}^{\infty}C_{j}\{ F_{k}(\phi_{j})\}^{2}\sum_{p=1}^{\infty}C_{p} \{ F_{\ell}(\phi_{p})\} ^{2}}\nonumber
\\ &\leq   \sup_{j \geq 1 }|C_{j}|
\sqrt{\sum_{j=1}^{\infty} \{ F_{k}(\phi_{j})\} ^{2}\sum_{p=1}^{\infty} \{ F_{\ell}(\phi_{p})\} ^{2}}
= \sup_{j \geq 1 }|C_{j}|\sqrt{N_{k}N_{\ell}}, 
\end{align}
\noindent where $\{F_{n}:n\geq 1\}$ has been introduced in Eq.~(\ref{normFn}), and satisfies \eqref{normB}--\eqref{ineq_norm}.
Under Assumption~A5, from Eq.~\eqref{cfm}, 
 \begin{equation*}
\|c\|_{B\times B}= \sup_{k, \ell \geq 1}|C(F_{k},F_{\ell})|\leq \sup_{k, \ell \geq 1 }\sup_{j \geq 1 }|C_{j}|\sqrt{N_{k}N_{\ell}}
=N\sup_{j \geq 1}|C_{j}|=N\|C\|_{\mathcal{L}(\widetilde{H})}.
\end{equation*}
This completes the proof. \hfill $\Box $

 \subsection{Proof of Lemma~\ref{lemsv}}

First observe that 
\begin{align*}
|C-C_{n}(F_{k})(F_{\ell})| 
&=\left| \sum_{j=1}^{\infty}C_{j}F_{k}(\phi_{n,j}^{\prime})F_{\ell}(\phi_{n,j}^{\prime})-C_{n,j}F_{k}(\phi_{n,j})F_{\ell}(\phi_{n,j})\right|\nonumber\\
&\leq  \sum_{j=1}^{\infty}|C_{j}||F_{k}(\phi_{n,j}^{\prime})||F_{\ell}(\phi_{n,j}^{\prime})-F_{\ell}(\phi_{n,j})|
+\sup_{j\geq 1}|C_{j}-C_{n,j}|
|F_{k}(\phi_{n,j}^{\prime})F_{\ell}(\phi_{n,j})|\nonumber\\
&\hspace*{6cm}+|C_{n,j}F_{\ell}(\phi_{n,j})||F_{k}(\phi_{n,j}^{\prime})-F_{k}(\phi_{n,j})|.
\end{align*}
Next note that the right-hand side is bounded above by
\begin{multline*}
 \sqrt{\sum_{j=1}^{\infty}C_{j} \{F_{k}(\phi_{n,j}^{\prime}) \}^{2}\sum_{j=1}^{\infty }C_{j} \{F_{\ell}(\phi_{n,j}^{\prime})-F_{\ell}(\phi_{n,j}) \}^{2}} \\
+ \sup_{ j \geq 1}|C_{j}-C_{n,j}|\sqrt{\sum_{j=1}^{\infty} \{F_{k}(\phi_{n,j}^{\prime}) \}^{2}
\sum_{j=1}^{\infty} \{F_{\ell}(\phi_{n,j}) \}^{2}}  \\
+\sqrt{\sum_{j=1}^{\infty}C_{n,j} \{F_{\ell}(\phi_{n,j}) \}^{2}\sum_{j=1}^{\infty }C_{n,j} \{F_{k}(\phi_{n,j}^{\prime})-F_{k}(\phi_{n,j}) \}^{2}} ,
\end{multline*}
and that this upper bound is itself smaller than
$$
 \sqrt{N_{k}}\sqrt{\sum_{j=1}^{\infty }C_{j} \{F_{\ell}(\phi_{n,j}^{\prime})-F_{\ell}(\phi_{n,j}) \}^{2}}
+\sup_{ j \geq 1 }|C_{j}-C_{n,j}|\sqrt{N_{k}}\sqrt{N_{\ell}}
 +\sqrt{N_{\ell}}\sqrt{\sum_{j=1}^{\infty }C_{n,j} \{F_{k}(\phi_{n,j}^{\prime})-F_{k}(\phi_{n,j}) \}^{2}}.
$$
Now, the latter expression can itself be bounded above by
$$
\max(N, \sqrt{N})\left[\sqrt{\|C\|_{\mathcal{L}(\widetilde{H})}\sum_{j=1}^{\infty } \{F_{\ell}(\phi_{n,j}^{\prime}-\phi_{n,j}) \}^{2}}  +\|C-C_{n}\|_{\mathcal{L}(\widetilde{H})} 
+\sqrt{\|C_{n}\|_{\mathcal{L}(\widetilde{H})}\sum_{j=1}^{\infty } \{F_{k}(
\phi_{n,j}^{\prime}-\phi_{n,j}) \}^{2}}\right]
$$
and hence, a fortiori, also by
\begin{multline*}
\max(N, \sqrt{N})\left[\|C-C_{n}\|_{\mathcal{L}(\widetilde{H})} +\sqrt{\|C\|_{\mathcal{L}(\widetilde{H})}\sum_{j=1}^{\infty }\sum_{m=1}^{\infty}
 \{F_{\ell}(\phi^{\prime}_{n,m}) \}^{2} \{\langle\phi_{n,j}^{\prime},\phi_{n,m}^{\prime}\rangle_{\widetilde{H}}-\langle\phi_{n,j},\phi_{n,m}^{\prime}\rangle_{\widetilde{H}}
 \}^{2}}\right. \\ \left. +\sqrt{\|C_{n}\|_{\mathcal{L}(\widetilde{H})}\sum_{j=1}^{\infty }\sum_{m=1}^{\infty}
 \{F_{k}(\phi^{\prime}_{n,m}) \}^{2} \{\langle\phi_{n,j}^{\prime},\phi_{n,m}^{\prime}\rangle_{\widetilde{H}}-\langle\phi_{n,j},\phi_{n,m}^{\prime}\rangle_{\widetilde{H}} \}^{2}}\right].
\end{multline*}
The latter expression can be rewritten as
\begin{multline*}
\max(N, \sqrt{N})\left[\|C-C_{n}\|_{\mathcal{L}(\widetilde{H})} +\sqrt{\|C\|_{\mathcal{L}(\widetilde{H})}\sum_{m=1}^{\infty}
 \{F_{\ell}(\phi^{\prime}_{n,m}) \}^{2}\sum_{j=1}^{\infty } \{\langle\phi_{n,j}^{\prime},\phi_{n,m}^{\prime}\rangle_{\widetilde{H}}-\langle\phi_{n,j},\phi_{n,m}^{\prime}\rangle_{\widetilde{H}} \}^{2}} \right. \\
\left. +\sqrt{\|C_{n}\|_{\mathcal{L}(\widetilde{H})}\sum_{m=1}^{\infty}
 \{F_{k}(\phi^{\prime}_{n,m}) \}^{2}\sum_{j=1}^{\infty } \{\langle\phi_{n,j}^{\prime},\phi_{n,m}^{\prime}\rangle_{\widetilde{H}}-\langle\phi_{n,j},\phi_{n,m}^{\prime}\rangle_{\widetilde{H}} \}^{2}}\right],
\end{multline*}
or equivalently as
\begin{multline*}
\max(N, \sqrt{N})\left[\|C-C_{n}\|_{\mathcal{L}(\widetilde{H})} 
+\sqrt{\|C\|_{\mathcal{L}(\widetilde{H})}\sum_{m=1}^{\infty}
\{F_{\ell}(\phi^{\prime}_{n,m}) \}^{2}\sum_{j=1}^{\infty }
 \{\langle\phi_{n,j},\phi_{n,m}\rangle_{\widetilde{H}}-\langle\phi_{n,j},\phi_{n,m}^{\prime}\rangle_{\widetilde{H}}
 \}^{2}} \right. \\ 
\left. +\sqrt{\|C_{n}\|_{\mathcal{L}(\widetilde{H})}\sum_{m=1}^{\infty}
 \{F_{k}(\phi^{\prime}_{n,m}) \}^{2}\sum_{j=1}^{\infty } \{\langle\phi_{n,j},\phi_{n,m}\rangle_{\widetilde{H}}-\langle\phi_{n,j},\phi_{n,m}^{\prime}\rangle_{\widetilde{H}} \}^{2}}\right],
\end{multline*}
and also as
\begin{multline*}
\max(N, \sqrt{N}) {\Bigg [}
\|C-C_{n}\|_{\mathcal{L}(\widetilde{H})} 
 +\sqrt{\|C\|_{\mathcal{L}(\widetilde{H})}\sum_{m=1}^{\infty}
 \{F_{\ell}(\phi^{\prime}_{n,m}) \}^{2}\|\phi_{n,m}-\phi_{n,m}^{\prime}\|_{\widetilde{H}}^{2}}  \\
 +\sqrt{\|C_{n}\|_{\mathcal{L}(\widetilde{H})}\sum_{m=1}^{\infty}
\{F_{k}(\phi^{\prime}_{n,m}) \}^{2}\|\phi_{n,m}-\phi_{n,m}^{\prime}\|_{\widetilde{H}}^{2}}
{\Bigg ]}.
\end{multline*}
It is now easy to see that the latter expression is bounded above by
\begin{multline*}
\max(N, \sqrt{N})\left\{
\|C-C_{n}\|_{\mathcal{L}(\widetilde{H})} 
+\sup_{m \geq 1} |F_{\ell}(\phi^{\prime}_{n,m}) |\sqrt{\|C\|_{\mathcal{L}(\widetilde{H})}\sum_{m=1}^{\infty}\|\phi_{n,m}-\phi_{n,m}^{\prime}\|_{\widetilde{H}}^{2}}\right. \\
\left.+\sup_{m \geq 1} | F_{k}(\phi^{\prime}_{n,m}) |\sqrt{\|C_{n}\|_{\mathcal{L}(\widetilde{H})}\sum_{m=1}^{\infty}\|\phi_{n,m}-\phi_{n,m}^{\prime}\|_{\widetilde{H}}^{2}}\right\}.
\end{multline*}
Recapping, we can then conclude that $|C-C_{n}(F_{k})(F_{\ell})|$ is smaller than the above term, and hence also
\begin{multline}
|C-C_{n}(F_{k})(F_{\ell})|  \leq \max(N, \sqrt{N}) {\Bigg [}
\|C-C_{n}\|_{\mathcal{L}(\widetilde{H})} +\max\left(\sqrt{\|C\|_{\mathcal{L}(\widetilde{H})}},\sqrt{\|C_{n}\|_{\mathcal{L}(\widetilde{H})}}\right)  \\
\left\{ \sup_{m \geq 1} |F_{\ell}(\phi^{\prime}_{n,m}) |+\sup_{m \geq 1} |F_{k}(\phi^{\prime}_{n,m}) |\right\}
\sqrt{\sum_{m=1}^{\infty }\|\phi_{n,m}
-\phi_{n,m}^{\prime}\|_{\widetilde{H}}^{2}} {\Bigg ]}.
\label{eqineempconkernel}
\end{multline}
 
Note that, under Assumption~A5, from Eq.~\eqref{inclusion2}, we have, for every integer $k \geq 1,$
\begin{eqnarray}
\sup_{m \geq 1} |F_{k}(\phi^{\prime}_{n,m}) |<\infty.
\label{ineqecant}
\end{eqnarray}
 Thus, considering $k_{n},$ as given in {Assumption~A2},  from Lemma \ref{lemmanew}, under (\ref{eqcondindisp}), applying 
Eq.~\eqref{fth2},    as $n \to \infty,$ we get
 \begin{align}
\sum_{m=1}^{\infty}\|\phi_{n,m}-
\phi_{n,m}^{\prime}\|_{\widetilde{H}}^{2}
& =\sum_{m=1}^{k_{n}}\|\phi_{n,m}-
\phi_{n,m}^{\prime}\|_{\widetilde{H}}^{2}+\sum_{m=k_{n}+1}^{\infty}\|\phi_{n,m}-
\phi_{n,m}^{\prime}\|_{\widetilde{H}}^{2}\nonumber\\
&\leq k_{n}\sup_{1\leq m\leq k_{n}}\|\phi_{n,m}-
\phi_{n,m}^{\prime}\|_{\widetilde{H}}^{2}+\sum_{m=k_{n}+1}^{\infty}\|\phi_{n,m}-
\phi_{n,m}^{\prime}\|_{\widetilde{H}}^{2}\nonumber\\
&\leq k_{n}8\Lambda_{k_{n}}^{2}
\|C_{n}-C\|_{\mathcal{L}(\widetilde{H})}^{2}+\sum_{m=k_{n}+1}^{\infty}\|\phi_{n,m}-
\phi_{n,m}^{\prime}\|_{\widetilde{H}}^{2}
 \label{ineqecant2}\\
 &\leq 
k_{n}8\Lambda_{k_{n}}^{2}
\|C_{n}-C\|_{\mathcal{S}(\widetilde{H})}^{2}+\sum_{m=k_{n}+1}^{\infty}\|\phi_{n,m}-
\phi_{n,m}^{\prime}\|_{\widetilde{H}}^{2}\to_{{\rm a.s.}} 0.\label{ffineq}
\end{align}

From Eqs.~(\ref{eqineempconkernel})--(\ref{ffineq}),
considering that under {Assumption~A5}, we have
$$
\sup_{k\geq 1}\sup_{m \geq 1} |F_{k}(\phi^{\prime}_{n,m}) |<\infty,
$$
we conclude that $\|c-c_{n}\|_{B\times B}=\sup_{k, \ell \geq 1}|(C-C_{n})(F_{k})(F_{\ell})|\to_{{\rm a.s.}} 0$ as $n\to \infty.$ \hfill $\Box $
 
 \subsection{Proof of Lemma~\ref{leminfinit}}

Let us first consider the following string of almost sure equalities:
\begin{equation}
C_{n,j} ( \phi_{n,j} - \phi_{n,j}^{\prime} ) = C_n ( \phi_{n,j} ) - C_{n,j} (\phi_{n,j}^{\prime}) =  ( C_n - C  )  ( \phi_{n,j} ) +  C ( \phi_{n,j} - \phi_{n,j}^{\prime}  ) +  ( C_j - C_{n,j}  ) \phi_{n,j}^{\prime}.\label{firstopres1}
\end{equation}
From Eq.~\eqref{firstopres1}, keeping in mind {Assumption~A2}, we can write
\begin{align}
\left\| \phi_{n,j} - \phi_{n,j}^{\prime} \right\|_B & \leq  
\frac{1}{C_{n,j}} \, \| ( C_n - C ) ( \phi_{n,j} ) \|_B  +  \frac{1}{C_{n,j}} \,  \|C  ( \phi_{n,j} - \phi_{n,j}^{\prime}  ) \|_B 
+  \frac{1}{C_{n,j}} \, \| ( C_j - C_{n,j}  ) \phi_{n,j}^{\prime} \|_B\nonumber \\
& =  \frac{1}{C_{n,j}} \, ( S_{1}+S_{2}+S_{3} ) ,\quad \mbox{{\rm a.s.}}
\label{eq28bb}
\end{align}

For sufficiently large $n$,  from Lemmas~\ref{lemma:3}--\ref{lemsv},  applying the Cauchy--Schwarz inequality, we get
\begin{align*}
S_{1}= \|\left( C_n - C \right) ( \phi_{n,j} ) \|_B  &= \sup_{m\geq 1}\left|\sum_{k=1}^{\infty}C_{n,k}F_{m}(\phi_{n,k})\langle
\phi_{n,k},\phi_{n,j}\rangle_{\widetilde{H}}-\sum_{k=1}^{\infty}C_{k}F_{m}(\phi_{n,k}^{\prime})\langle
\phi_{n,k}^{\prime},\phi_{n,j}\rangle_{\widetilde{H}}\right| \\
&= \sup_{m\geq 1 }\left|\sum_{k=1}^{\infty}\sum_{\ell=1}^{\infty}t_{\ell}F_{\ell}(\phi_{n,j}) \{C_{n,k}F_{m}(\phi_{n,k})F_{\ell}(\phi_{n,k})-C_{k}F_{m}(\phi_{n,k}^{\prime})F_{\ell}(\phi_{n,k}^{\prime}) \}\right| \\
&= \sup_{m\geq 1}\left|\sum_{\ell = 1}^{\infty}t_{\ell}F_{\ell}(\phi_{n,j})\sum_{k=1}^{\infty}C_{n,k}F_{m}(\phi_{n,k})F_{\ell}(\phi_{n,k})-C_{k}F_{m}(\phi_{n,k}^{\prime})F_{\ell}(\phi_{n,k}^{\prime})\right|
\end{align*}
for every integer $ j \geq 1$. Now the right-hand side can be bounded above by
$$
\sup_{m\geq 1 }\sqrt{\sum_{\ell=1}^{\infty}t_{\ell} \{ F_{\ell}(\phi_{n,j})\} ^{2}}
\times \sqrt{\sum_{\ell=1}^{\infty}t_{\ell}\left\{\sum_{k=1}^{\infty}C_{n,k}F_{m}(\phi_{n,k})F_{\ell}(\phi_{n,k})-C_{k}F_{m}(\phi_{n,k}^{\prime})F_{\ell}(\phi_{n,k}^{\prime})\right\}^{2}} ,
$$
which in turn, is smaller than
$$
\|\phi_{n,j}\|_{\widetilde{H}}\sqrt{\sum_{\ell = 1}^{\infty}t_{\ell}}\sup_{m, \ell}\left|\sum_{k=1}^{\infty}C_{n,k}F_{m}(\phi_{n,k})F_{\ell}(\phi_{n,k})-C_{k}F_{m}(\phi_{n,k}^{\prime})F_{\ell}(\phi_{n,k}^{\prime})\right|
=\|c_{n}-c\|_{B\times B}.
$$
Thus $S_1$ is smaller than the latter expression, from which we can then deduce that
\begin{multline}
S_1 \leq 
\max(N, \sqrt{N}) {\Bigg [}
\|C-C_{n}\|_{\mathcal{L}(\widetilde{H})} 
+2\max\left(\sqrt{\|C\|_{\mathcal{L}(\widetilde{H})}},\sqrt{\|C_{n}\|_{\mathcal{L}(\widetilde{H})}}\right)\left\{ \sup_{\ell \geq 1}\sup_{m \geq 1} |F_{\ell}(\phi^{\prime}_{n,m}) |\right\}
\\
 \times \sqrt{k_{n}8\Lambda_{k_{n}}^{2}
\|C_{n}-C\|_{\mathcal{L}(\widetilde{H})}^{2}+\sum_{m=k_{n}+1}^{\infty}\|\phi_{n,m}-
\phi_{n,m}^{\prime}\|_{\widetilde{H}}^{2}} \Bigg ].
\label{eq28_}
\end{multline}
Similarly,
\begin{align}
S_{2}= \|C  ( \phi_{n,j} - \phi_{n,j}^{\prime} ) \|_B & =\sup_{m \geq 1 }\left|\sum_{k=1}^{\infty}\sum_{\ell = 1}^{\infty}
t_{\ell}C_{k}F_{m}(\phi_{n,k}^{\prime})F_{\ell}(\phi_{n,k}^{\prime})
F_{\ell} ( \phi_{n,j} - \phi_{n,j}^{\prime} )\right| \nonumber\\
&\leq 
\sup_{m\geq 1}\sqrt{\sum_{\ell = 1}^{\infty}t_{\ell} \{F_{\ell} ( 
\phi_{n,j} - \phi_{n,j}^{\prime} ) \}^{2}}\sqrt{\sum_{\ell = 1}^{\infty}t_{\ell}
\left\{\sum_{k=1}^{\infty} C_{k}F_{m}(\phi_{n,k}^{\prime})F_{\ell}(\phi_{n,k}^{\prime})\right\}^{2}}\nonumber\\
&\leq 
\|\phi_{n,j} - \phi_{n,j}^{\prime}\|_{\widetilde{H}}\sup_{m,\ell \geq 1}\left|\sum_{k=1}^{\infty} C_{k}F_{m}(\phi_{n,k}^{\prime})F_{\ell}(\phi_{n,k}^{\prime})\right|\nonumber\\
&= \|\phi_{n,j} - \phi_{n,j}^{\prime}\|_{\widetilde{H}} \times \|c\|_{B\times B}
\leq  \|\phi_{n,j} - \phi_{n,j}^{\prime}\|_{\widetilde{H}} \times N \times \|C\|_{\mathcal{S}(\widetilde{H})}  \mbox{ {\rm a.s.}}
\label{eqcc}
\end{align}
Moreover, we see that under {Assumption~A5},
\begin{eqnarray}
S_{3}
\leq  \sup_{ j \geq 1 }|C_j - C_{n,j}| \left\|\phi_{n,j}^{\prime}\right\|_B
\leq  V\|C-C_{n}\|_{\mathcal{L}(\widetilde{H})}\leq  V\|C-C_{n}\|_{\mathcal{S}(\widetilde{H})}  \mbox{ {\rm a.s.}} 
\label{eqs3}
\end{eqnarray}
In addition, it follows from Lemma~\ref{theorem2} that  
$\|C_{n}-C\|_{\mathcal{S}(\widetilde{H})}\to_{{\rm a.s.}} 0,$ and, from Lemma \ref{lem1}, $C_{n,j}\to_{{\rm a.s.}} C_{j}$ as $n\to \infty$.  For $\varepsilon=C_{k_{n}}/2,$ we can thus find $n_{0}$ such that for $n\geq n_{0},$
\begin{align}
\|C_{n}-C\|_{\mathcal{L}(\widetilde{H})}\leq \varepsilon & =C_{k_{n}}/2 \mbox{   {\rm a.s.},} \quad 
 | C_{n,k_{n}}-C_{k_{n}} |\leq \widetilde{\varepsilon}  \leq \|C_{n}-C\|_{\mathcal{L}(\widetilde{H})}, \nonumber\\ 
 C_{n,k_{n}}\geq C_{k_{n}}-\widetilde{\varepsilon} & \geq C_{k_{n}}-\|C_{n}-C\|_{\mathcal{L}(\widetilde{H})}
\geq C_{k_{n}}-C_{k_{n}}/2\geq C_{k_{n}}/2.
\label{cvzerohh}
\end{align}

From Eqs.~(\ref{eq28bb})--(\ref{eqs3}), and for $n$ large enough to ensure that Eq.~\eqref{cvzerohh} holds,  the following almost surely inequalities are then satisfied for all $j \in \{ 1, \ldots,  k_{n}\}$:
\begin{multline*}
\sup_{j \in \{ 1, \ldots, k_{n}\}} \| \phi_{n,j} - \phi_{n,j}^{\prime} \|_B  \leq \frac{1}{C_{n,k_{n}}} {\Bigg [}
\max(N, \sqrt{N}) {\Bigg [}
\|C-C_{n}\|_{\mathcal{L}(\widetilde{H})} 
+ 2\max\left(\sqrt{\|C\|_{\mathcal{L}(\widetilde{H})}},\sqrt{\|C_{n}\|_{\mathcal{L}(\widetilde{H})}}\right)\left\{\sup_{\ell \geq 1}\sup_{m \geq 1} |F_{\ell}(\phi^{\prime}_{n,m}) |\right\} \\
\times \sqrt{k_{n}8\Lambda_{k_{n}}^{2}
\|C_{n}-C\|_{\mathcal{L}(\widetilde{H})}^{2}+\sum_{m=k_{n}+1}^{\infty}\|\phi_{n,m}-
\phi_{n,m}^{\prime}\|_{\widetilde{H}}^{2}} {\Bigg ]}   \\
+\sup_{j \in \{ 1, \ldots, k_{n}\}}\|\phi_{n,j} - \phi_{n,j}^{\prime}\|_{\widetilde{H}} \, N \, \|C\|_{\mathcal{S}(\widetilde{H})}+ 
V\|C-C_{n}\|_{\mathcal{S}(\widetilde{H})} {\Bigg ]}.
\end{multline*}
and, from Eq.~\eqref{cvzerohh},  the right-hand term is bounded above by
\begin{multline*}
\frac{2}{C_{k_{n}}} {\Bigg [} \max(N, \sqrt{N}) {\Bigg [}
\|C-C_{n}\|_{\mathcal{L}(\widetilde{H})} +2\max\left(\sqrt{\|C\|_{\mathcal{L}(\widetilde{H})}},\sqrt{\|C_{n}\|_{\mathcal{L}(\widetilde{H})}}\right) \left\{ \sup_{\ell \geq 1}\sup_{m \geq 1} |F_{\ell}(\phi^{\prime}_{n,m}) | \right\}  \\
\times \sqrt{k_{n}8\Lambda_{k_{n}}^{2}
\|C_{n}-C\|_{\mathcal{L}(\widetilde{H})}^{2}+\sum_{m=k_{n}+1}^{\infty}\|\phi_{n,m}-
\phi_{n,m}^{\prime}\|_{\widetilde{H}}^{2}} {\Bigg ]} \\
+\sup_{j \in \{ 1, \ldots, k_{n}\}}\|\phi_{n,j} - \phi_{n,j}^{\prime}\|_{\widetilde{H}}N\|C\|_{\mathcal{S}(\widetilde{H})}+ 
V\|C-C_{n}\|_{\mathcal{S}(\widetilde{H})} {\Bigg ]} \mbox{ \rm a.s.} 
\end{multline*}
Hence, Eq.~\eqref{eqlem7} holds. The almost sure convergence to zero directly follows from Lemma~\ref{theorem2}, under \eqref{orl8}. \hfill $\Box $

 \subsection{Proof of Lemma~\ref{ultlemma}}

The following  identities  are considered:
\begin{equation}
\sum_{j=1}^{k_{n}}\langle \rho(x),\phi_{n,j}\rangle_{\widetilde{H}}\phi_{n,j}-\sum_{j=1}^{k_{n}}\langle \rho(x),\phi_{n,j}^{\prime }\rangle_{\widetilde{H}}\phi_{n,j}^{\prime } =\sum_{j=1}^{k_{n}}\langle \rho(x),\phi_{n,j}\rangle_{\widetilde{H}}(\phi_{n,j}-\phi_{n,j}^{\prime })+
\sum_{j=1}^{k_{n}}\langle \rho(x),
\phi_{n,j}-\phi_{n,j}^{\prime }\rangle_{\widetilde{H}}
\phi_{n,j}^{\prime }.
\label{decomposition}
\end{equation}
From Eq.~\eqref{decomposition}, applying the Cauchy--Schwarz inequality, under {Assumption~A3}, we find
\begin{multline*}
\sup_{x\in B,\ \|x\|_{B}\leq 1}\left\|\sum_{j=1}^{k_{n}}\langle \rho(x),\phi_{n,j}\rangle_{\widetilde{H}}\phi_{n,j}-\sum_{j=1}^{\infty}\langle \rho(x),\phi_{n,j}^{\prime }\rangle_{\widetilde{H}}\phi_{n,j}^{\prime }\right\|_{B}\nonumber\\
\leq \sup_{x\in B,\ \|x\|_{B}\leq 1}\sum_{j=1}^{k_{n}}\|\rho(x)\|_{\widetilde{H}} \times \|\phi_{n,j}
\|_{\widetilde{H}} \times \|\phi_{n,j}-\phi_{n,j}^{\prime }\|_{B}
+\|\rho(x)\|_{\widetilde{H}} \times \|\phi_{n,j}-\phi_{n,j}^{\prime }\|_{\widetilde{H}}
\times \|\phi_{n,j}^{\prime }\|_{B} \\
+\sup_{x\in B,\ \|x\|_{B}\leq 1}\left\|\sum_{j=k_{n}+1}^{\infty }\langle \rho(x),\phi_{n,j}^{\prime }\rangle_{\widetilde{H}}\phi_{n,j}^{\prime }\right\|_{B}.
\end{multline*}
Now the right-hand side can be bounded above by
$$
\sup_{x\in B,\ \|x\|_{B}\leq 1}\|\rho(x)\|_{\widetilde{H}}\left(\sum_{j=1}^{k_{n}}\|\phi_{n,j}-\phi_{n,j}^{\prime }\|_{B}+\|\phi_{n,j}-\phi_{n,j}^{\prime }\|_{B}\sup_{j \geq 1}\|\phi_{n,j}^{\prime }\|_{B}\right) 
+\sup_{x\in B,\ \|x\|_{B}\leq 1}\left\|\sum_{j=k_{n}+1}^{\infty }\langle \rho(x),\phi_{n,j}^{\prime }\rangle_{\widetilde{H}}\phi_{n,j}^{\prime }\right\|_{B},
$$
and the latter expression is smaller than
$$
\sup_{x\in B,\ \|x\|_{B}\leq 1}\|\rho\|_{\mathcal{L}(\widetilde{H})}\|x\|_{\widetilde{H}}
(1+V)\sum_{j=1}^{k_{n}}
\|\phi_{n,j}-\phi_{n,j}^{\prime }\|_{B}
+\sup_{x\in B,\ \|x\|_{B}\leq 1}\left\|\sum_{j=k_{n}+1}^{\infty }\langle \rho(x),\phi_{n,j}^{\prime }\rangle_{\widetilde{H}}\phi_{n,j}^{\prime }\right\|_{B}.
$$
One can then conclude because this last expression is itself smaller than
$$
 \|\rho\|_{\mathcal{L}(\widetilde{H})}(1+V)\sum_{j=1}^{k_{n}}
\|\phi_{n,j}-\phi_{n,j}^{\prime }\|_{B}
+\sup_{x\in B,\ \|x\|_{B}\leq 1}\left\|\sum_{j=k_{n}+1}^{\infty }\langle \rho(x),\phi_{n,j}^{\prime }\rangle_{\widetilde{H}}\phi_{n,j}^{\prime }\right\|_{B},
$$
which tends a.s. to $0$ as $n \to \infty$. This concludes the proof. \hfill $\Box$

\section{ARB$(1)$  estimation and prediction: Strong consistency results}
\label{sec:3}
For every $x\in 
B\subset \widetilde{H},$ the following component-wise estimator $\widetilde{\rho}_{k_n}$ of $\rho$ will be considered: 
\begin{equation*}
\widetilde{\rho}_{k_n} (x) = \left( \widetilde{\Pi}^{k_n} D_n C_{n}^{-1} \widetilde{\Pi}^{k_n} \right) (x) =  \left\{ \displaystyle \sum_{j=1}^{k_n} \frac{1}{C_{n,j}} \langle x, \phi_{n,j} \rangle_{\widetilde{H}}\widetilde{\Pi}^{k_n} D_n(\phi_{n,j}) \right\}, 
\end{equation*}
where $\widetilde{\Pi}^{k_n}$ has been introduced in Eq.~\eqref{proy}, and  $C_{n},$ $C_{n,j},$ $\phi_{n,j}$ and $D_{n}$  have been defined in equations  
\eqref{empcn}--(\ref{empdn}), respectively. 
\begin{theorem}
\label{thmain}
As before, let $X$ be a standard ARB$(1)$  process. Under the conditions of Lemmas \ref{leminfinit} and \ref{ultlemma} (see Remark~\ref{remprev}), for all $\eta >0,$ 
\begin{equation*}
\Pr \left(\|\widetilde{\rho}_{k_{n}}-\rho \|_{\mathcal{L}(B)}\geq \eta\right)\leq \mathcal{K}\exp ( - {n\eta^{2}}/{Q_{n}} ),
\end{equation*}
where 
$$
Q_{n}=\mathcal{O}\left\{ \left(  C_{k_{n}}^{-1}k_{n}\sum_{j=1}^{k_{n}}a_{j}\right)^{2}\right\}
$$ 
as $n \to \infty$. Therefore, if 
$$
k_{n}C_{k_{n}}^{-1}\sum_{j=1}^{k_{n}}a_{j}= o\big\{ \sqrt{{n}/{\ln(n)}}\big\},
$$
as $n\to \infty$, then also, 
$\|\widetilde{\rho}_{k_{n}}-\rho \|_{\mathcal{L}(B)}\to_{a.s} 0$.
\end{theorem}

\bigskip
\noindent
\textit{Proof.} 
For every $x\in B,$ such that $\|x\|_{B}\leq 1,$ applying the triangle inequality, under Assumptions A1 and A2,
\begin{equation}
\| \widetilde{\Pi}^{k_{n}}D_{n}C_{n}^{-1}\widetilde{\Pi}^{k_{n}}(x)-\widetilde{\Pi}^{k_{n}}\rho\widetilde{\Pi }^{k_{n}}(x)\|_{B} 
\leq \| \widetilde{\Pi}^{k_{n}}(D_{n}-D)C_{n}^{-1}
\widetilde{\Pi }^{k_{n}}(x)\|_{B}+\|\widetilde{\Pi}^{k_{n}}
(DC_{n}^{-1}-\rho)\widetilde{\Pi}^{k_{n}}(x)\|_{B}=S_{1}(x)+S_{2}(x).
\label{eq1mthh}
\end{equation}
Under the conditions assumed in Lemma \ref{ultlemma}, considering inequality \eqref{cvzerohh},
\begin{align}
S_{1}(x)&=\| \widetilde{\Pi}^{k_{n}}(D_{n}-D)C_{n}^{-1}
\widetilde{\Pi }^{k_{n}}(x)\|_{B} \leq \left\|C_{n,k_{n}}^{-1}\sum_{j=1}^{k_{n}}\sum_{p=1}^{k_{n}}\langle x,\phi_{n,j}\rangle_{\widetilde{H}}\langle (D_{n}-D)(\phi_{n,j}),\phi_{n,p} \rangle_{\widetilde{H}}\phi_{n,p}
\right\|_{B}
\nonumber\\
&\leq \left|C_{n,k_{n}}^{-1}\right|\sum_{j=1}^{k_{n}}\sum_{p=1}^{k_{n}} |\langle x,\phi_{n,j}\rangle_{\widetilde{H}}
| \times |\langle (D_{n}-D)(\phi_{n,j}),\phi_{n,p} \rangle_{\widetilde{H}} | \times \|\phi_{n,p}
\|_{B} 
\leq 2C_{k_{n}}^{-1} k_{n}\|D_{n}-D\|_{\mathcal{L}(\widetilde{H})}\sum_{p=1}^{k_{n}} \|\phi_{n,p}
\|_{B}
\nonumber\\&
\leq 2VC_{k_{n}}^{-1}k_{n}^{2}
\|D_{n}-D\|_{\mathcal{S}(\widetilde{H})}.\label{eq2mth}
\end{align}
Furthermore, applying the triangle inequality, we find
\begin{align}
S_{2}(x)&= \|\widetilde{\Pi}^{k_{n}}
(DC_{n}^{-1}-\rho)\widetilde{\Pi}^{k_{n}}(x)\|_{B}\nonumber\\
&\leq \|\widetilde{\Pi}^{k_{n}}
DC_{n}^{-1}\widetilde{\Pi}^{k_{n}}(x)-\widetilde{\Pi}^{k_{n}}DC^{-1}\Pi^{k_{n}}(x)\|_{B} +\|\widetilde{\Pi}^{k_{n}}DC^{-1}\Pi^{k_{n}}(x)-
\widetilde{\Pi}^{k_{n}}\rho \widetilde{\Pi}^{k_{n}}(x)\|_{B} =S_{21}(x)+S_{22}(x).
\label{eq3mth}
\end{align}

Under {Assumptions A1}--A2, $C^{-1}$ and $C_{n}^{-1}$ are bounded on the subspaces generated by $\{\phi_{j}: j \in \{ 1,\dots,k_{n} \} \}$ and $\{\phi_{n,j}:  j \in \{ 1, \ldots, k_{n} \}\},$ respectively.      Consider now
\begin{align*}
S_{21}(x) & = \|\widetilde{\Pi}^{k_{n}}
DC_{n}^{-1}\widetilde{\Pi}^{k_{n}}(x)-\widetilde{\Pi}^{k_{n}}DC^{-1}\Pi^{k_{n}}(x)\|_{B} \\
& = \left\|\sum_{j=1}^{k_{n}}\sum_{p=1}^{k_{n}}\frac{1}{C_{n,j}}\langle x,\phi_{n,j}-\phi_{n,j}^{\prime }\rangle_{\widetilde{H}}\langle D(\phi_{n,j}),\phi_{n,p}\rangle_{\widetilde{H}} \, \phi_{n,p}\right.  \left.+\sum_{j=1}^{k_{n}}\sum_{p=1}^{k_{n}}\left(\frac{1}{C_{n,j}}-\frac{1}{C_{j}}\right)\langle x,\phi_{n,j}^{\prime }\rangle_{\widetilde{H}}\langle D(\phi_{n,j}),\phi_{n,p}\rangle_{\widetilde{H}}\, \phi_{n,p}\right. \\
& \quad \quad \quad \quad \left.+\sum_{j=1}^{k_{n}}\sum_{p=1}^{k_{n}}\frac{1}{C_{j}}\langle x,
\phi_{n,j}^{\prime }\rangle_{\widetilde{H}}\langle D(\phi_{n,j}-\phi_{n,j}^{\prime }),\phi_{n,p}\rangle_{\widetilde{H}} \, \phi_{n,p}\right\|_{B}.
\end{align*}
Note that 
\begin{multline}
S_{21}(x)\leq \sum_{j=1}^{k_{n}}\sum_{p=1}^{k_{n}}
\left|\frac{1}{C_{n,k_{n}}}\right|\|\phi_{n,j}-\phi_{n,j}^{\prime }\|_{\widetilde{H}} \times \|D\|_{\mathcal{L}(\widetilde{H})} \times \|\phi_{n,p}\|_{B} \\
+\left|\frac{1}{C_{n,j}}-\frac{1}{C_{j}}\right|\|D\|_{\mathcal{L}(\widetilde{H})} \times \|\phi_{n,p}\|_{B} 
+\left|\frac{1}{C_{j}}\right|\|D\|_{\mathcal{L}(\widetilde{H})} \times \|\phi_{n,j}-\phi_{n,j}^{\prime }\|_{\widetilde{H}} \times
\|\phi_{n,p}\|_{B}.
\label{eq4mth}
\end{multline}

From Lemma 4.3 on p.~104 of \cite{Bosq00}, for every integer $ j \geq 1,$ under  {Assumption~A1},
\begin{eqnarray}
&&\|\phi_{n,j}-\phi_{n,j}^{\prime }\|_{\widetilde{H}}\leq 
a_{j}\|C_{n}-C\|_{\mathcal{L}(\widetilde{H})},\label{eq4mthb}
\end{eqnarray}
\noindent where  $a_{j}$ has been introduced in \eqref{a_j}.
Then, in Eq.~\eqref{eq4mth}, considering again  inequality \eqref{cvzerohh},  keeping in mind that $C_{j}^{-1}\leq a_{j},$  we obtain
\begin{equation}
S_{21}(x) \leq
5C_{k_{n}}^{-1}\sum_{p=1}^{k_{n}}\|\phi_{n,p}\|_{B} \times
\|D\|_{\mathcal{L}(\widetilde{H})} \times \|C_{n}-C\|_{\mathcal{L}(\widetilde{H})}\sum_{j=1}^{k_{n}}a_{j} 
\leq 
5Vk_{n}C_{k_{n}}^{-1}\|D\|_{\mathcal{L}(\widetilde{H})} \times \|C_{n}-C\|_{\mathcal{S}(\widetilde{H})}\sum_{j=1}^{k_{n}}a_{j}.
\label{eq5mth}
\end{equation}

Applying again the triangle and  Cauchy--Schwarz inequalities, we deduce from  (\ref{eq4mthb}) that 
\begin{align}
S_{22} &= \|\widetilde{\Pi}^{k_{n}}DC^{-1}\Pi^{k_{n}}(x)-
\widetilde{\Pi}^{k_{n}}\rho \widetilde{\Pi}^{k_{n}}(x)\|_{B}
\nonumber\\
&= \left\|\sum_{j=1}^{k_{n}}\sum_{p=1}^{k_{n}}\langle x,\phi_{n,j}^{\prime}-\phi_{n,j}\rangle_{\widetilde{H}}\langle \rho(\phi_{n,j}^{\prime}),\phi_{n,p}\rangle_{\widetilde{H}} \,
\phi_{n,p} +\langle x,\phi_{n,j}\rangle_{\widetilde{H}}
\langle \rho(\phi_{n,j}^{\prime}-\phi_{n,j}),\phi_{n,p}\rangle_{\widetilde{H}} \,
\phi_{n,p}
\right\|_{B}\nonumber\\
&\leq  \sum_{j=1}^{k_{n}}\sum_{p=1}^{k_{n}}\|x\|_{\widetilde{H}} \times \|\phi_{n,j}^{\prime}-\phi_{n,j}\|_{\widetilde{H}}
\times \|\rho\|_{\mathcal{L}(\widetilde{H})} \times \|\phi_{n,j}^{\prime}\|_{\widetilde{H}} \times 
\|\phi_{n,p}\|_{\widetilde{H}} \times \|\phi_{n,p}\|_{B} \nonumber\\
& \quad \quad \quad \quad +\|x\|_{\widetilde{H}} \times \|\phi_{n,j}\|_{\widetilde{H}} \times \|\rho\|_{\mathcal{L}(\widetilde{H})}
\times \|\phi_{n,j}^{\prime}-\phi_{n,j}\|_{\widetilde{H}} \times \|\phi_{n,p}\|_{\widetilde{H}} \times 
\|\phi_{n,p}\|_{B}\nonumber\\
&\leq  2\, \|\rho\|_{\mathcal{L}(\widetilde{H})} \times \|C_{n}-C\|_{\mathcal{S}(\widetilde{H})}\left(\sum_{p=1}^{k_{n}}\|\phi_{n,p}\|_{B}\right)
\left(\sum_{j=1}^{k_{n}}a_{j}\right)\nonumber\\
&\leq 2V\|\rho\|_{\mathcal{L}(\widetilde{H})} \times \|C_{n}-C\|_{\mathcal{S}(\widetilde{H})}
k_{n}\sum_{j=1}^{k_{n}}a_{j}.\label{eq6mth}
\end{align}

From Eqs.~(\ref{eq1mthh})--(\ref{eq6mth}),
\begin{multline}
\sup_{x\in B,\ \|x\|_{B}\leq 1}\| \widetilde{\Pi}^{k_{n}}D_{n}C_{n}^{-1}\widetilde{\Pi}^{k_{n}}(x)-\widetilde{\Pi}^{k_{n}}\rho\widetilde{\Pi }^{k_{n}}(x)\|_{B} \leq  2VC_{k_{n}}^{-1}k_{n}^{2}
\|D_{n}-D\|_{\mathcal{S}(\widetilde{H})}  \\
+ \|C_{n}-C\|_{\mathcal{S}(\widetilde{H})}2Vk_{n}\sum_{j=1}^{k_{n}}a_{j} ( 5/2C_{k_{n}}^{-1}\|D\|_{\mathcal{L}(\widetilde{H})}+\|\rho\|_{\mathcal{L}(\widetilde{H})}  ).
\label{eq7mth}
\end{multline}

 From Eq.~\eqref{eq7mth},  applying now Theorem 4.2 (p.~99) and Theorem~4.8 (p.~116) in~\cite{Bosq00}, we get, for any $\eta>0$, 
 $$
\Pr \left( \sup_{x\in B,\ \|x\|_{B}\leq 1}\| \widetilde{\Pi}^{k_{n}}D_{n}C_{n}^{-1}\widetilde{\Pi}^{k_{n}}(x)-\widetilde{\Pi}^{k_{n}}\rho\widetilde{\Pi }^{k_{n}}(x)\|_{B} > \eta \right) \hspace{8cm} 
 $$
\begin{align}
&\leq  \Pr \left\{ \sup_{x\in B,\ \|x\|_{B}\leq 1} S_{1}(x)> \eta \right\} +
\Pr \left\{ \sup_{x\in B,\ \|x\|_{B}\leq 1} S_{21}(x)+S_{22}(x)> \eta \right\} \nonumber\\
&\leq \Pr \left(\|D_{n}-D\|_{\mathcal{S}(\widetilde{H})}>
\frac{\eta}{2VC_{k_{n}}^{-1}k_{n}^{2}}\right)+\Pr \left(\|C_{n}-C\|_{\mathcal{S}(\widetilde{H})}>
\frac{\eta}{2Vk_{n}\sum_{j=1}^{k_{n}}a_{j}\left[ 5/2C_{k_{n}}^{-1}\|D\|_{\mathcal{L}(\widetilde{H})}+\|\rho\|_{\mathcal{L}(\widetilde{H})} \right]}\right)\nonumber\\
&\leq  8\exp\left[ -\frac{n\eta^{2}}{\left(2VC_{k_{n}}^{-1}k_{n}^{2}\right)^{2}\left[ \gamma+\delta \{ {\eta}/{(2VC_{k_{n}}^{-1}k_{n}^{2})} \}\right] }\right] +4\exp ( - {n\eta^{2}}/{Q_{n}} ),
\label{inecknbb}
\end{align}
with    $\gamma $ and $\delta$   being positive numbers,  depending on $\rho$ and $P_{\varepsilon_{0}},$  respectively, introduced in Theorems~4.2 and 4.8 of~\cite{Bosq00}. Here, 
\begin{equation}
Q_{n} = 
4V^{2}k_{n}^{2}\left(\sum_{j=1}^{k_{n}}a_{j}\right)^{2}
\left\{ 5/2C_{k_{n}}^{-1}\|D\|_{\mathcal{L}(\widetilde{H})}+\|\rho\|_{\mathcal{L}(\widetilde{H})} \right\}^{2}
\times \left\{\alpha_{1}+\beta_{1} \frac{\eta }{2Vk_{n}\sum_{j=1}^{k_{n}}a_{j}\left( 5/2C_{k_{n}}^{-1}\|D\|_{\mathcal{L}(\widetilde{H})}+\|\rho\|_{\mathcal{L}(\widetilde{H})} \right)}\right\}, 
\label{eq8mth}
\end{equation}
where again $\alpha_{1}$ and $\beta_{1}$ are positive constants depending on $\rho$ and $P_{\varepsilon_{0}},$ respectively.
From Eqs.~\eqref{inecknbb}--\eqref{eq8mth}, we see that if
  $k_{n}C_{k_{n}}^{-1}\sum_{j=1}^{k_{n}}a_{j} = o \big \{\sqrt{{n}/{\ln(n)}}\big\}$ as $n\to \infty,$
 then, the Borel--Cantelli Lemma, and  Lemma~\ref{ultlemma}  lead to the desired almost sure convergence to zero  (see also  Remarks \ref{remprev}--\ref{remprev2}).
 \hfill $\Box $

\begin{corollary}
\label{cor2}
Under the conditions of Theorem \ref{thmain},  
$\|\widetilde{\rho}_{k_{n}}(X_{n})-\rho(X_{n})\|_{B}\to_{{\rm a.s.}} 0$ as $n\to \infty$.
\end{corollary}
 The proof is straightforward from Theorem \ref{thmain} because 
$\|\widetilde{\rho}_{k_{n}}(X_{n})-\rho(X_{n})\|_{B}\leq \|\widetilde{\rho}_{k_{n}}-\rho\|_{\mathcal{L}(B)}
\|X_{0}\|_{B}\to_{a.s} 0$ as $n\to \infty$ under Assumption~A1.

\section{Examples: Wavelets in Besov and  Sobolev spaces}
\label{examples}

It is well known that wavelets provide orthonormal bases of 
$L^{2}(\mathbb{R}),$ and 
unconditional bases for several function spaces including Besov spaces, $B_{p,q}^{s}$ with $s\in \mathbb{R}$ and $1\leq p,q\leq \infty$. Sobolev or H\"older spaces constitute  interesting  particular cases of Besov spaces; see, e.g., \cite{Triebel83}.

Consider now orthogonal wavelets on the interval $[0,1].$ Adapting wavelets
to a finite interval requires some modifications as described in \cite{Cohen}. 
Let $s>0$ for an $([s]+1)$-regular Multiresolution Analysis  (MRA) of $L^{2}[0,1]$, where $[\cdot]$ stands for the integer part. The father and mother wavelets $\varphi$ and $\psi$  are such that  $\varphi ,\psi \in \mathcal{C}^{[s]+1}[0,1].$ Also $\varphi $ and  its derivatives, up to order $[s]+1,$ have a fast decay; see Corollary~5.2 in \cite{Daubechies}.

Given $J$ such that $2^J \geq 2([s]+1),$ the construction in \cite{Cohen} starts from a finite set of  $2^J$ scaling functions $\{ \varphi_{J,k}: k \in \{0,\ldots,2^J - 1\} \}$. For each $j\geq J,$ a set $2^j$ wavelet functions  $\{ \psi_{j,k}: k \in \{ 0, ,\ldots,2^j - 1 \} \}$ are also considered. The collection of these
functions  $\{ \varphi_{J,k}: k\in \{ 0, \ldots,2^J - 1 \} \}$ and $\{ \psi_{j,k}: k \in \{0,\ldots,2^j - 1 \}\}$ with $j\geq J$ form a complete orthonormal system of $L^{2} [0,1]$. The associated reconstruction formula is given, for all $t\in [0,1]$ and $f \in L^{2} [0,1] $, by
\begin{equation}
f(t) = \displaystyle \sum_{k=0}^{2^J-1} \alpha_{J,k}^{f} \varphi_{J,k} (t) + \displaystyle \sum_{j=J}^\infty \displaystyle \sum_{k=0}^{2^{j}-1} \beta_{j,k}^{f} \psi_{j,k}(t), 
\label{eqwavtrans}
\end{equation}
where 
$$
\alpha_{J,k}^{f} =   \int_{0}^{1} f(t) \overline{\varphi_{J,k}(t)}dt
$$
for all $k \in \{0,\ldots,2^J - 1\}$ and 
$$
 \beta_{j,k}^{f} =  \int_{0}^{1} f(t) \overline{\psi_{j,k}(t)}dt 
$$
for all $k \in\{0,\dots, 2^{j}-1\}$ with $ j\geq J$.

 The Besov spaces $B_{p,q}^{s}[0,1]$ can be characterized in terms of wavelets coefficients. Specifically,  denote by   $\mathcal{S}^{\prime }$  the dual of $\mathcal{S},$ the Schwarz space. Then $f\in \mathcal{S}^{\prime }$ belongs to $B_{p,q}^{s}[0,1]$ for some $s\in \mathbb{R}$ and $1\leq p,q\leq \infty$ if and only if
\begin{equation}
\|f\|_{p,q}^{s}\equiv \|\varphi*f\|_{p}+\left\{ \sum_{j=1}^{\infty}\left(2^{js}\| \psi_{j}*f\|_{p}\right)^{q}\right\}^{1/q}<\infty.
\label{enormbesovwav}
\end{equation}
For $\beta >1/2,$ consider a self-adjoint positive  operator $\mathcal{T}:H_{2}^{-\beta}[0,1]\longrightarrow H_{2}^{\beta}[0,1]$ on $L^{2}[0,1]$ belonging to the unit ball of trace operators on $L^{2}[0,1].$ Assume that     $\mathcal{T}:H_{2}^{-\beta}[0,1]\longrightarrow H_{2}^{\beta}[0,1]$ and $\mathcal{T}^{-1}:H_{2}^{\beta}[0,1]\longrightarrow 
H_{2}^{-\beta}[0,1]$ are bounded linear operators. In particular, there exists an orthormal basis $\{v_{k}: k \geq 1\}$   of $L^{2}[0,1]$ such that, for every $\ell \geq 1,$,
$\mathcal{T}(v_{\ell})=t_{\ell}v_{\ell},$ with $\sum_{\ell \geq 1}t_{\ell}=1.$

In what follows, consider a wavelet basis $\{v_{\ell}: \ell \geq 1\}$, and define the kernel $t$ of  $\mathcal{T}$, for all $s,t\in [0,1]$, by
$$
t(s,t)=\frac{1}{2^{J}}\sum_{k=0}^{2^{J-1}}\varphi_{J,k}(s)\varphi_{J,k}(t)+
\frac{2^{2\beta}-1}{2^{2\beta (1-J)}}\sum_{j=J}^\infty  \sum_{k=0}^{2^{j}-1}2^{-2j\beta }\psi_{j,k}(s)\psi_{j,k}(t).
$$
In Lemma~\ref{lemma:1},  $(F_{\mathbf{m}})= \{F^{\varphi}_{J,k}: k\in \{ 0,\dots,2^{J}-1\} \}\cup \{F_{j,k}^{\psi}: k \in \{0,\dots, 2^{j}-1,$ $j\geq J\}\}$ are then defined as follows:
  \begin{align}
  F^{\varphi}_{J,k} &= \varphi_{J,k},\quad k \in \{ 0,\dots, 2^{J}-1 \} \nonumber\\
  F_{j,k}^{\psi} &= \psi_{j,k},\quad k \in \{0,\dots, 2^{j}-1\},\quad j\geq J.
  \label{id1lem1}
  \end{align}
Furthermore, the sequence $(t_{\mathbf{m}}) = \{ t^{\varphi}_{J,k}: k \in \{ 0,\dots,2^{J}-1 \} \}\cup \{ t^{\psi}_{j,k}: k
\in\{ 0,\dots, 2^{j}-1\},  j\geq J\}$  involved in the definition of the inner product in $\widetilde{H},$ is given by
\begin{align}
t^{\varphi}_{J,k} &= \frac{1}{2^{J}}, \quad k \in \{ 0,\dots,2^{J-1}\} .\nonumber\\
t^{\psi}_{j,k} &= \frac{2^{2\beta}-1}{2^{2\beta (1-J)}} \, 2^{-2j\beta }, \quad k \in \{ 0,\dots,2^{j-1}\},\quad  j\geq J.
\label{tnseqlem1}
\end{align}

In view of Proposition~2.1 in~\cite{Angelini03}, the choice (\ref{id1lem1})--(\ref{tnseqlem1}) of 
$(F_{\mathbf{m}})$ and $(t_{\mathbf{m}})$ leads to the definition of $\widetilde{H}=
\{ H_{2}^{\beta }[0,1]\} ^{\star }=H_{2}^{-\beta}[0,1],$ constituted by the restriction to $[0,1]$ of the tempered distributions  $g\in \mathcal{S}^{\prime }(\mathbb{R}),$ such that
$(I-\Delta )^{-\beta /2}g\in L^{2}(\mathbb{R}),$ with $(I-\Delta )^{-\beta /2}$ denoting the Bessel potential of order $\beta $; see \cite{Triebel83}. 

Now define $B=B_{\infty,\infty}^{0}([0,1],)$ and  $B^{\star}=
 B^{0}_{1,1}[0,1].$ From Eq.~\eqref{enormbesovwav}, the corresponding norms, in term of the discrete wavelet transform introduced in Eq.~\eqref{eqwavtrans}, are respectively given, for every $f \in B$, by
$$
\left\| f \right\|_{B} =  \displaystyle \sup \left\{ |\alpha_{J,k}^{f} |, k \in \{0,\dots, 2^{J-1}\};  |\beta_{j,k}^{f}  |, k\in \{0,\dots, 2^{j}-1\},  j\geq J \right\} 
$$
and for every $g \in B^{\star}$, by
$$
\left\| g  \right\|_{B^{\star}} =  \displaystyle \sum_{k=0}^{2^{J}-1} | \alpha_{J,k}^{ g}  |+ \displaystyle \ \sum_{j=J}^{\infty }\sum_{k=0}^{2^{j}-1} | \beta_{j,k}^{ g} |. 
$$
Therefore, 
$$
 B^{\star}=B_{1,1}^{0}[0,1]\hookrightarrow H=L^{2}[0,1] \hookrightarrow B=B_{\infty,\infty}^{0}\hookrightarrow 
 \widetilde{H}= H_{2}^{-\beta}[0,1].
$$
Also, for $\beta >1/2,$ $\widetilde{H}^{\star}=H^{\beta}[0,1]\hookrightarrow B^{\star}=B_{1,1}^{0}[0,1].$ 

For $\gamma>2\beta ,$  consider the operator $C=(I-\Delta )^{-\gamma  },$   i.e., given by the $2\gamma/\beta $ power of the Bessel potential of order $\beta ,$ restricted to $L^{2}[0,1]$.  From spectral theorems on spectral calculus stated, e.g., in \cite{Triebel83}, we have that, for every $f\in C^{1/2}\{ H^{-\beta}[0,1]\}$,
\begin{align*}
\|f\|_{\mathcal{H}(X)}^{2}=
\langle C^{-1}(f),f\rangle_{H^{-\beta }[0,1]}
&= \langle (I-\Delta )^{-\beta /2}\{C^{-1}(f)\},(I-\Delta )^{-\beta /2}\left(f\right)\rangle_{L^{2}[0,1]}
\nonumber\\
&=\sum_{j=1}^{\infty }f_{j}^{2}\lambda_{j}\{ (I-\Delta )^{(\gamma -\beta )}\} \geq \sum_{j=1}^{\infty }f_{j}^{2}\lambda_{j}\{(I-\Delta )^{ \beta }\} =\|f\|_{H^{\beta }[0,1]}^{2}=\|f\|_{\widetilde{H}^{\star}}^{2},
\end{align*}
where, for every integer $j\geq 1,$
$$
f_{j}=\int_{0}^{1}f(s)\phi_{j}(s)ds,
$$ 
with $\{ \phi_{j}:  j \geq 1 \}$ denoting the eigenvectors of the Bessel potential $(I-\Delta )^{-\beta/2}$ of order $\beta ,$ restricted to $L^{2}[0,1],$  and $\{\lambda_{j}\{(I-\Delta )^{\gamma -\beta}\}:  j \geq 1 \}$ being the eigenvalues of $(I-\Delta )^{-\beta}C^{-1}$ on $L^{2}[0,1]$. Thus, {Assumption~A4} holds.
 Furthermore, from Embedding Theorems between fractional Sobolev spaces
 (see \cite{Triebel83}), {Assumption~A5} also holds, under the condition $\gamma >2\beta > 1,$ considering  $H=L^{2}[0,1]$. 
 
  \section{Final comments} 
  \label{fc}
  
 Section~\ref{examples} illustrates the motivation of the presented approach in relation to functional prediction in nuclear spaces.  Specifically, the current literature on ARB$(1)$  prediction has been developed for $B=\mathcal{C}[0,1],$ the space of continuous functions on $[0,1],$ with the sup norm (see, e.g., \cite{AlvarezLi16, Bosq00}), and  $B=\mathcal{D}[0,1]$ consisting of the right-continuous functions on $[0,1]$ having a left limit at every $t\in [0,1],$ with the  Skorokhod topology; see, e.g., \cite{Hajj11}. 
 This paper provides a more flexible framework, where functional prediction can be performed, in a consistent way, for instance, in nuclear spaces,  as follows from the continuous inclusions showed in   Section~\ref{examples}.
  \noindent Note that the two above-referred usual  Banach spaces, $\mathcal{C}[0,1]$ and $\mathcal{D}[0,1],$ are included in the Banach space $B$ considered in  Section~\ref{examples}; see the Online Supplement for the results of a simulation study. 

\section*{Acknowledgments}

This work was supported in part by project MTM2015--71839--P
(co-funded by Feder funds), of the DGI, MINECO, Spain.

\end{document}